\RequirePackage{fix-cm}

\documentclass{article}

\usepackage{amsmath}
\usepackage{amssymb}
\usepackage{amsthm}

\usepackage{hyperref,enumerate}
\usepackage{tikz}

\usepackage{xcolor}

\usepackage[version=3]{mhchem}
	\renewcommand{\longleftrightarrow}{\ce{<=>}}

\usepackage[sort&compress,square,comma,numbers]{natbib}

\newcommand{\intclr}{M}

\newcommand{\cane}{{\mathbf e}}

\newcommand{\pry}{{\widehat y}}

\newcommand{\rone}{r1}
\newcommand{\rtwo}{r2}

\newcommand{\Nrone}{r3}

\DeclareRobustCommand{\comment}[1]{}

\newcommand{\mytriangle}{\hfill $\triangle$}
\newcommand{\mydiamond}{\hfill $\Diamond$}

\DeclareMathOperator{\supp}{supp}

\DeclareRobustCommand{\mybox}{\hfill $\Box$}
\DeclareRobustCommand {\mydiamond}{\hfill $\Diamond$}
\DeclareRobustCommand{\mytriangle}{\hfill $\triangle$} 

\DeclareRobustCommand{\calA}{{\cal A}}
\DeclareRobustCommand{\calB}{{\cal B}}
\DeclareRobustCommand{\calC}{{\cal C}}
\DeclareRobustCommand{\calD}{{\cal D}}
\DeclareRobustCommand{\calE}{{\cal E}}

\DeclareRobustCommand{\calJ}{{\cal J}}

\DeclareRobustCommand{\calO}{{\cal O}}
\DeclareRobustCommand{\calP}{{\cal P}}
\DeclareRobustCommand{\calR}{{\cal R}}
\DeclareRobustCommand{\calS}{{\cal S}}

\DeclareRobustCommand{\calX}{{\cal X}}
\DeclareRobustCommand{\calY}{{\cal Y}}

\DeclareRobustCommand{\n}{{\mathbb N}}
\DeclareRobustCommand{\r}{{\mathbb R}}
    \DeclareRobustCommand{\rplus}{\r_{\geqslant 0}}
    \DeclareRobustCommand{\rpluss}{\r_{> 0}}
\DeclareRobustCommand{\z}{{\mathbb Z}}
    \DeclareRobustCommand{\zplus}{\z_{\geqslant 0}}
    \DeclareRobustCommand{\zpluss}{\z_{> 0}}

\theoremstyle{definition}
\newtheorem{definition}{Definition}

\theoremstyle{remark}
\newtheorem{example}[definition]{Example}
\newtheorem{remark}[definition]{Remark}

\theoremstyle{plain}
\newtheorem{lemma}[definition]{Lemma}
\newtheorem{proposition}[definition]{Proposition}
\newtheorem{theorem}[definition]{Theorem}
\newtheorem{corollary}[definition]{Corollary}

\title{Intermediates, Catalysts, Persistence, and Boundary Steady States}
\author{Michael Marcondes de Freitas, Elisenda Feliu, and Carsten Wiuf \\ {\small Department of Mathematical Sciences, University of Copenhagen}}
\date{\today}

\begin{document}

\maketitle

\abstract{For dynamical systems arising from chemical reaction networks, persistence is the property that each species concentration remains positively bounded away from zero, as long as species concentrations were all positive in the beginning. We describe two graphical procedures for simplifying reaction networks without breaking known necessary or sufficient conditions for persistence, by iteratively removing so-called intermediates and catalysts from the network. The procedures are easy to apply and, in many cases, lead to highly simplified network structures, such as monomolecular networks. For specific classes of reaction networks, we show that these conditions for persistence are equivalent to one another. Furthermore, they can also be characterized by easily checkable strong connectivity properties of a related graph. In particular, this is the case for (conservative) monomolecular networks, as well as cascades of a large class of post-translational modification systems (of which the MAPK cascade and the $n$-site futile cycle are prominent examples). Since one of the aforementioned sufficient conditions for persistence precludes the existence of boundary steady states, our method also provides a graphical tool to check for that.}

\vskip 1ex


\vskip 1ex


\tableofcontents

\section{Introduction}

Since the seminal works of Horn, Jackson and Feinberg in the 70's (\cite{feinberg-1979,gunawardena-2003,horn--jackson-1972}, and references therein), chemical reaction network theory (CRNT) has provided a fruitful framework to study the dynamical systems describing how the concentrations of the involved chemical species evolve over time. Of great interest has been the long-term behavior of these systems, for example, whether they may exhibit oscillatory behavior \cite{feinberg-1987}, local asymptotic stability {\cite{anderson-2008,anderson-2011,feinberg-1987,gopalkrishnan--miller--shiu-2014,sontag-2001}}, or persistence \cite{angeli--deleenheer--sontag-2007c,angeli--deleenheer--sontag-2011,craciun--nazarov--pantea-2013,deshpande--gopalkrishnan-2014,gnacadja-2011i,gnacadja-2011,gopalkrishnan--miller--shiu-2014}.

The mathematical concept of persistence models the property that every spe\-cies concentration remains above a certain threshold, as long as there were positive amounts of each species in the beginning. Besides its intrinsic relevance to the applied sciences, most notably in population biology \cite{smith--thieme-2011}, the concept of persistence has also drawn attention in the context of CRNT on account of its connection with the global attractor conjecture \cite{gopalkrishnan--miller--shiu-2014}.

It can be difficult to determine if the solutions to a system of ordinary differential equations are persistent case by case. A recent contribution was given by Angeli, De Leenheer and Sontag \cite{angeli--deleenheer--sontag-2007c}, who provided two checkable conditions, one sufficient, and the other one necessary, for the persistence of conservative reaction networks. Their sufficient conditions were further developed and relaxed by Deshpande and Gopalkrishnan in \cite{deshpande--gopalkrishnan-2014}. These criteria work under fairly general assumptions on the reaction kinetics. But perhaps unsurprisingly, reaction networks become more difficult to analyze the larger they are, often times exponentially so \cite{cordone--ferrarini--piroddi-2005}. Thus, criteria for persistence in terms of a simplified ``skeleton'' of the given network are desirable. More importantly, simplified versions retaining the properties of interest of the original network may also give insight into the underlying biological mechanism, suggesting what might be the leading causes of the presence (or absence) of said properties. For example, for the class of post-translational modification (PTM) systems of Thomson and Gunawardena \cite{thomson--gunawardena-2009}, {or cascades of PTM systems}, persistence can be characterized in terms of strong connectedness of the underlying substrate network {at each layer} {of the cascade}, as we shall see.

That is the motivation for our model simplification approach to study persistence. In this work we describe a process through which one may simplify a reaction network by iteratively removing ``intermediates'' \cite{feliu--wiuf-2013}, and/or ``catalysts.'' Intuitively speaking, an intermediate is a transient species appearing in the middle of a chain of reactions. Catalysts, on the other hand, are reactants which remain unchanged in every reaction, except possibly for interactions exclusively with other catalysts. Our main contribution is to show that the removal of intermediates and/or catalysts does not break the conditions for persistence given in \cite{angeli--deleenheer--sontag-2007c} {and \cite{deshpande--gopalkrishnan-2014}}. Our main results in this work may be informally stated as follows.

\begin{theorem}\label{thm:informal_main}
The conditions for persistence of reaction networks in \cite{angeli--deleenheer--sontag-2007c} and \cite{deshpande--gopalkrishnan-2014} are invariant under the removal of intermediate species.
\end{theorem}

\begin{theorem}\label{thm:informal_main2}
The conditions for persistence of reaction networks in \cite{angeli--deleenheer--sontag-2007c} and \cite{deshpande--gopalkrishnan-2014} are invariant under the removal of catalysts.	
\end{theorem}

\begin{theorem}\label{thm:informal_main3}
The same minimally simplified reaction network is always obtained by iteratively removing intermediates and catalysts until none can be found, independently of the order in which they are removed.
\end{theorem}

As shown by various examples throughout this work taken from the systems biology literature, reaction networks naturally exhibit many intermediate complexes and catalysts. So, their removal will often reduce dramatically the size of the network, facilitating its inspection for persistence. To illustrate this, consider a simple one-site phosphorylation process, which can be modeled by the reaction network
\begin{equation}\label{eq:prototype}
	\begin{aligned}
		E + S_0 \longleftrightarrow ES_0 \longrightarrow E + S_1\qquad 
		F + S_1 \longleftrightarrow FS_1 \longrightarrow F + S_0,
	\end{aligned}
\end{equation}
where $S_0, S_1$ represent, respectively, the dephosphorylated and phosphorylated forms of a substrate, $E$ acts as a kinase, $F$ acts as a phosphatase, and $ES_0$ and $FS_1$ are intermediate {protein complexes} in the phosphorylation/dephosphorylation mechanism. Using our results, one may show that necessary or sufficient conditions for persistence for (\ref{eq:prototype}) are a consequence of the same necessary or sufficient conditions for its much simpler underlying substrate model,
\begin{equation}\label{eq:substrate}
	S_0 \longleftrightarrow S_1\,.
\end{equation}
For monomolecular models such as (\ref{eq:substrate}), the necessary or sufficient conditions for persistence are actually equivalent, and, furthermore, characterized by the strong connectedness of each connected component. In fact, (\ref{eq:prototype}) will turn out to be a special case of PTM system. 

We emphasize that iteratively removing intermediates and catalysts---and, if eventually obtaining a monomolecular network, then checking it for strong connectedness of its connected components---is essentially a graphical procedure.

This paper is organized as follows. In Section \ref{sec:preliminaries}, we review the basic formalism of reaction networks. We present the conditions for persistence in \cite{angeli--deleenheer--sontag-2007c} and \cite{deshpande--gopalkrishnan-2014} in the form we shall use in this work, and discuss their relationship with boundary steady states. A few trivial but notable examples we shall refer to several times throughout the work are given, and persistence is characterized for monomolecular networks in terms of strong connectedness of its connected components. In Section \ref{sec:operations}, we define the concepts of intermediates and catalysts. We describe the networks obtained from their removal, and formally state our main results (Theorems \ref{thm:main}, \ref{thm:main2} and \ref{thm:primitive}), concerning how these operations do not break the aforementioned conditions for persistence. Some biologically relevant examples are presented in Section \ref{sec:examples}, the most important of which being cascades of a class of PTM systems. In Section \ref{sec:proofofmainresult} we return to our main results, giving the details of the {proofs}. A short appendix with some auxiliary technical results is presented at the end.

\section{Reaction Networks}\label{sec:preliminaries}

In what follows we denote the set of nonnegative real (respectively, integer) numbers by $\rplus$ (respectively, $\zplus)$, and denote the set of strictly positive real (respectively, integer) numbers by $\rpluss$ (respectively, $\zpluss$). 
{We denote the boundary of the nonnegative orthant by $\partial \rplus^n$.}
Given $x \in \r^n$, for some $n \in \zpluss$, we write $x \geqslant 0$ to mean that $x \in \rplus^n$, that is, each coordinate of $x$ is nonnegative. We write $x > 0$ to mean that $x \geqslant 0$, and at least one coordinate of $x$ is positive, and write $x \gg 0$ to mean that $x \in \rpluss^n$, in other words, each coordinate of $x$ is strictly positive. For any finite set $X$, the notation $|X|$ represents the number of elements of $X$. Given $n \in \zpluss$, we write $[n] := \{1, \ldots, n\}$. By convention $[0] := \varnothing$.

\subsection{Basic Formalism}

In this work we take the approach of defining reaction networks from their reaction graphs. Thus, a {\em reaction network} is an ordered triple $G = (\calS, \calC, \calR)$ in which $\calS$ is a finite, possibly empty set, $\calC$ is a finite subset of $\rplus^n$, where $n := |\calS|$, and $(\calC, \calR)$ is a digraph with no self-loops. The set $\calS$ is called the {\em species} set of the reaction network. Its elements are tacitly assumed to be ordered in some fixed way, say,
\[
	\calS = \{S_1, \ldots, S_n\}\,.
\]
We identify the elements $(\alpha_1, \ldots, \alpha_n)$ of $\calC$, called the {\em complexes} of the reaction network, with the formal linear combinations of species
\[
	\alpha_1S_1 + \cdots + \alpha_nS_n\,.
\]
The digraph $(\calC, \calR)$ is called the {\em reaction graph} of $G$, and its edges are referred to as the {\em reactions} of the network. We further assume that each complex takes part in at least one reaction, and that each species is part of at least one complex. Formally, this means that each vertex of $(\calC, \calR)$ has indegree or outdegree at least one, and that for each $i \in [n]$, there exists $(\alpha_1, \ldots, \alpha_n) \in \calC$ such that $\alpha_i > 0$. It follows that {$\calS = \varnothing \Leftrightarrow \calC = \varnothing \Leftrightarrow \calR = \varnothing$, in which case the network} is referred to as the {\em empty reaction network}.

The reactions are also tacitly assumed to be ordered in some fixed way, say,
\[
	\calR = \{R_1, \ldots, R_m\}\,,
\]
where $m := |\calR|$. We often express the reaction 
\[
R_j =  \big((\alpha_{1j}, \ldots, \alpha_{nj}), (\alpha'_{1j}, \ldots, \alpha'_{nj}) \big)
\]
as
\[
	R_j\colon
	\quad
	\sum_{i = 1}^n \alpha_{ij}S_i \longrightarrow \sum_{i = 1}^n \alpha'_{ij}S_i\,,
	\quad
	j = 1, \ldots, m\,.
\]
The complex on the lefthand side is referred to as the {\em reactant} of the reaction, while the complex on the righthand side is referred to as its {\em product}. The species $S_i$ such that $\alpha_{ij} > 0$ are, accordingly, called the {\em reactants} of $R_j$, while the species $S_i$ for which $\alpha'_{ij} > 0$ are called the {\em products} of the reaction.

A {\em reaction path} in $G$ is a directed path in the digraph $(\calC, \calR)$, that is, a sequence of reactions 
$$ y_0\rightarrow y_1\rightarrow \dots \rightarrow y_{k-1} \rightarrow y_k$$
such that $y_{j-1} \rightarrow y_j \in \calR\,$ for all $j \in [k]$ and all complexes are different.
Similarly,  an {\em undirected reaction path} in $G$ is a path in the undirected graph underlying  $(\calC, \calR)$.
In this case, we write $y_0 \ \textrm{---} \ y_1 \ \textrm{---} \ \cdots \ \textrm{---} \ y_k$, where each `---' can either be `$\leftarrow$' or `$\rightarrow$' in $(\calC, \calR)$.
By abuse of terminology, we refer to the connected components of the reaction graph $(\calC, \calR)$ as the connected components of $G$.

With the above notation, we   define the $n \times m$ matrix $N$,
\[
	N_{ij} := \alpha'_{ij} - \alpha_{ij}\,,
	\quad
	i = 1, \ldots, n\,,
	\quad
	j = 1, \ldots, m\,,
\]
known as the {\em stoichiometric matrix} of the network. The column-space of $N$, which is a subset of $\r^n$, is called the {\em stoichiometric subspace of $G$}, and denoted by $\Gamma$. The sets $(s_0 + \Gamma) \cap \rplus^n$, $s_0 \in \rplus^n$, are called the {\em stoichiometric compatibility classes} of $G$. Let
	\[
		Q_j := \{i \in [n]\,|\ \alpha_{ij} > 0\}\,,
		\quad
		j = 1, \ldots, m\,,
	\]
be the subset of indices corresponding to the reactants of $R_j$.

The system of differential equations governing the evolution of the concentrations of the species of the network is given by
\begin{equation}\label{eq:crntode} 
\frac{ds}{dt} = Nr(s(t))\,,
\quad
t \geqslant 0\,,
\quad
s \geqslant 0\,,
\end{equation}
where $r\colon \rplus^n \rightarrow \rplus^m$ is a vector-valued function modeling the kinetic rates of each reaction as functions of the reactant species, henceforth referred to simply as the {\em vector of reaction rates}. We shall assume throughout this work that the vector of reaction rates satisfies the following hypotheses:

\begin{itemize}
	\item[(\rone)] $r = (r_1, \ldots, r_{m})\colon \calO \rightarrow \r^{m}$ is continuously differentiable on a neighborhood $\calO$ of $\r_{\geqslant 0}^{n}$, and $r(s) \geqslant 0$ for every $s \geqslant 0$.

	\item[(\rtwo)] For each $j \in [m]$, and for each $s = (s_1, \ldots, s_{n}) \in \r_{\geqslant 0}^{n}$,
	\[
		r_j(s) = 0
		\quad
		\Leftrightarrow
		\quad
		s_i = 0
		\quad
		\text{for some} \ i \in Q_j\,.
	\]

	\item[(\Nrone)] The flow of (\ref{eq:crntode}) is forward-complete; in other words, for any initial state, the (unique) maximal solution of the corresponding initial value problem in (\ref{eq:crntode}) is defined for all $t \geqslant 0$.
\end{itemize}

We note that (r1)--{(r2)} are satisfied under the most common kinetic assumptions in the literature, namely, mass-action, or more general power-law kinetics, Michaelis-Menten kinetics, or Hill kinetics, as well as combinations of these \cite[pages 585--586]{angeli--deleenheer--sontag-2010}. 
We also note that it follows from (r1) and (\rtwo) that  the non-negative and positive  orthants, $\r_{\geqslant 0}^{n}$ and $\r_{> 0}^n$, are forward invariant for the flow of (\ref{eq:crntode}) (see, for instance, \cite[Section VII]{Sontag:2001} or \cite[Section 16]{Amann}).

We will often give a reaction network by simply listing all the reactions in the network. When we do so, the sets of species and complexes will be tacitly implied. For instance,
\[
	G\colon
	\quad
	S_1 + S_2 \longleftrightarrow S_3 \longrightarrow S_1 + S_4
\]
is the reaction network $G = (\calS, \calC, \calR)$ obtained by setting
\[
	\calS := \{S_1, S_2, S_3, S_4\},\qquad 
	\calC := \{S_1 + S_2, S_3, S_1 + S_4\}
\]
and
\[
	\calR := \{S_1 + S_2 \rightarrow S_3, S_3 \rightarrow S_1 + S_3, S_3 \rightarrow S_1 + S_4\}
\]
in the formalism above.

\begin{definition}[Implied Subnetworks]\label{def:implied_network}
	Let $G = (\calS, \calC, \calR)$ be a reaction network, and $\calE \subseteq \calS$ be a subset of species. We define the {\em subnetwork implied by $\calE$} as the network $G_\calE = (\calS_\calE, \calC_\calE, \calR_\calE)$ consisting of reactions of $G$ which involve exclusively species in $\calE$. More precisely, $\calR_\calE \subseteq \calR$ is the subset of reactions
	\[
		\sum_{i = 1}^n \alpha_{i}S_i \longrightarrow \sum_{i = 1}^n \alpha'_{i}S_i
	\]
such that {$\alpha_i = \alpha'_i = 0$ for every $i \in [n]$ such that $S_i \notin \calE$}. We then define $\calC_\calE \subseteq \calC$ to be the subset of complexes that appear as reactant or product of some reaction in $\calR_\calE$. Finally, $\calS_\calE \subseteq \calS$ is defined as the subset of species which are part of some complex in $\calC_\calE$.
\mytriangle	
\end{definition}
Although it is always true that $\calS_\calE \subseteq \calE$, it may be the case that $\calS_\calE \neq \calE$. To see this, consider the reaction network $G$ with  $\calR=\{S_1 + S_2 \longrightarrow S_3 + S_4, S_4 \longrightarrow S_2\}$ and set $\calE := \{S_1, S_2, S_4\}$. Then $G_\calE$ consists of the reaction $S_4 \longrightarrow S_2$.
In particular, $\calS_\calE = \{S_2, S_4\} \subsetneq \{S_1, S_2, S_4\} = \calE$.

\subsection{Siphons, P- and T-Semiflows, Drainable Sets and Self-Replicable Sets}

A few more concepts pertaining to reaction networks are needed. Some of the terminology below is adapted from Petri net theory. See \cite{angeli--deleenheer--sontag-2007c} for the context. But since no results from Petri net theory itself are needed, we chose to define these concepts as directly pertaining to their respective reaction networks, rather than the Petri nets associated with them.

\begin{definition}[Siphons]\label{def:cnrt-siphons}
	A nonempty subset of species $\Sigma \subseteq \calS$ is called a {\em siphon} if every reaction which has a product in $\Sigma$ also has a reactant in $\Sigma$. A siphon is said to be {\em minimal} if it does not properly contain any other siphon.
	\mytriangle
\end{definition}

\begin{example}[Single Phosphorylation Mechanism]
	The minimal siphons of the single phosphorylation mechanism from the Introduction (\ref{eq:prototype}) are $\{E, ES_0\}$, $\{F, FS_1\}$, and $\{S_0, S_1, ES_0, FS_1\}$.
	\mydiamond
\end{example}

\begin{remark}\label{rk:siphonpath}
Let $y \rightarrow y_1 \rightarrow \cdots \rightarrow y_k \rightarrow y'$ be a reaction path in a reaction network $G$, and suppose $\Sigma$ is a siphon containing some species $S'$ that is part of $y'$. Then each of the complexes $y, y_1,\dots,y_k$ must have at least one of its species in $\Sigma$.
\mybox
\end{remark}

Given a vector $\omega = (\omega_1, \ldots, \omega_n) \in \rplus^n$ associated with the species set $\calS$ of a reaction network $G = (\calS, \calC, \calR)$, its {\em support} is defined to be the subset of species $\supp \omega := \{S_i \in \calS\,|\ \omega_{i} > 0\}$. Similarly, given a vector $v = (v_1, \ldots, v_m) \in \rplus^m$ associated with the reaction set $\calR$ of $G$, its {\em support} is defined to be the subset of reactions $\supp v := \{R_j \in \calR\,|\ v_j > 0\}$. Although we use the same notation in both cases, it will be clear from the context whether the underlying vector is associated with the species or the reaction set.

\begin{definition}[P- and T-Semiflows]
	A {\em P-semiflow} or {\em positive conservation law} of a reaction network is any nonzero vector $\omega \in \rplus^{n}$ such that $\omega^TN = 0$. We say that a reaction network is {\em conservative} if it has a strictly positive P-semiflow $\omega \gg 0$, that is, $\supp \omega = \calS$. A {\em T-semiflow} of a reaction network is any nonzero vector $v \in \rplus^m$ such that $Nv = 0$. We say that a reaction network is {\em consistent} if it has a strictly positive T-semiflow $v \gg 0$, that is, if $\supp v = \calR$.
	\mytriangle
\end{definition}

\begin{definition}[Siphon/P-Semiflow Property]	We say that a reaction network has the {\em siphon/P-semiflow property} if every siphon contains the support of a P-semiflow.	
\mytriangle	
\end{definition}

Nonempty sets of species not containing the support of a P-semiflow are also known in the literature as {\em critical} \cite{deshpande--gopalkrishnan-2014}. So, a reaction network has the siphon/P-semiflow property if, and only if every siphon is noncritical.

Note that, since every siphon is either itself minimal, or else {contains} a minimal siphon, we need only check {whether every minimal siphon contains the support of a P-semiflow}. We give a couple more trivial examples. Besides further illustrating the scope of the concepts just introduced, they will be used several times in the analysis of more elaborate examples further down.

\begin{example}[Empty Networks]\label{ex:emptynetwork}
	Our formalism allows for reaction networks to be empty. Any such network is vacuously conservative, consistent, and also has the siphon/P-semiflow property.
\mydiamond
\end{example}

\begin{example}[Inflows]\label{ex:inflows} 
	Consider a reaction network $G = (\calS, \calC, \calR)$. If one can find a reaction path in $G$ of the form
	\[
		0 \longrightarrow y_1 \longrightarrow \cdots \longrightarrow y_k,
	\]
	then {by Remark~\ref{rk:siphonpath}}, none of the species that are a part of any of the complexes $y_1, \ldots, y_k$  belongs to a siphon {because no species is part of the complex $0$}. This observation may drastically reduce the number of species one is concerned about in checking the siphon/P-semiflow property.
	
	In particular, if $G$ is such that $0 \rightarrow S \in \calR$ for each $S \in \calS$, then $G$ has no siphons. {In this case, $G$ has vacuously the siphon/P-semiflow property.}
\mydiamond
\end{example}

We next introduce the concepts of drainable and self-replicable siphons. In \cite[Definition 3.1(2--3)]{deshpande--gopalkrishnan-2014}, these concepts were defined in terms of ``$G$-reaction pathways.'' 
We show in Proposition \ref{prop:understood} in Subsection \ref{subsec:equivalence} in the appendix that both definitions are equivalent. This equivalence is already implicitly used in the proofs of the results in  \cite{deshpande--gopalkrishnan-2014}.

\begin{definition}[Drainable and Self-Replicable Sets]\label{def:d_or_sr}
	Let $G = (\calS, \calC, \calR)$ be a reaction network. A nonempty subset of species $\Sigma \subseteq \calS$ is said to be {\em drainable} if there exists a sequence of reactions $y_1 \rightarrow y_1', \ldots, y_k \rightarrow y_k' \in \calR$  such that
	\[
	\left( \sum_{j = 1}^k (y_j' - y_j) \right)_{\hspace{-0.1cm} i}    < 0\,,
		\quad
		\forall i \in [n]\, \colon \ S_i \in \Sigma\,.
	\]
	If there exists one such a sequence of reactions such that
	\[
		\left( \sum_{j = 1}^k (y_j' - y_j) \right)_{\hspace{-0.1cm} i}    > 0\,,
		\quad
		\forall i \in [n]\, \colon \ S_i \in \Sigma\,,
	\]
	then $\Sigma$ is said to be {\em self-replicable}. In either case, the reactions need not be pairwise distinct.
\mytriangle	
\end{definition}

We summarize some properties of critical, drainable and self-replicable siphons we will need further down.

\begin{proposition}\label{prop:connection}
Let $G = (\calS, \calC, \calR)$ be a reaction network, and $\Sigma \subseteq \calS$ a nonempty subset. Then,
	\begin{itemize}
		\item[{\em(}i\,{\em)}] if $\Sigma$ is drainable or self-replicable, then it is critical; and
		
		\item[{\em(}ii\,{\em)}] if $\Sigma$ is a minimal critical siphon, then it is drainable or self-replicable.
	\end{itemize}

\end{proposition}

\begin{proof}
	See \cite[Theorem 5.3]{deshpande--gopalkrishnan-2014}.
\end{proof}

\begin{corollary}\label{cor:equivcrit}
A reaction network $G$ has the siphon/P-semiflow property if and only if $G$ does not have any drainable or self-replicable siphons.
\end{corollary}

\subsection{Persistence and Boundary Steady States}\label{subsec:persistence}

The existence of drainable siphons and the siphon/P-semiflow property are linked to persistence and the existence of boundary steady states. The connection is made precise in this subsection, where we compile results from \cite{deshpande--gopalkrishnan-2014,angeli--deleenheer--sontag-2007c,shiu--sturmfels-2010}.

Intuitively, persistence (of a reaction network) is the property that no species concentration goes below a certain threshold as the system evolves, as long as they were initially all positive. This threshold may depend on the initial conditions though. In order to formulate this more precisely, let $\sigma\colon \rplus \times \rplus^n \rightarrow \rplus^n$ be the {semiflow} of (\ref{eq:crntode}). In other words, for each initial state $s_0 \in \rplus^n$, $\sigma(\cdot, s_0) \colon \rplus \rightarrow \rplus^n$ is the unique, solution of (\ref{eq:crntode}). The solution is unique in virtue of (r1), and defined for all $t \geqslant 0$ on account of (r3).

\begin{definition}[Persistence]\label{def:persistence}
	A reaction network (\ref{eq:crntode}) is said to be {\em persistent} if
	\begin{equation}\label{eq:persistent}
		\liminf_{t \to \infty} \sigma_i(t, s_0) > 0\,,
		\quad
		\forall i \in [n]\,,
	\end{equation}
	for every initial state $s_0 \gg 0$.
\mytriangle	
\end{definition}

We also introduce a weaker notion of persistence. First, recall that, for each $s_0 \geqslant 0$, the {\em $\omega$-limit set of $s_0$} is the set
	\[
		\omega(s_0) := \bigcap_{\tau \geqslant 0} \overline{\bigcup_{t \geqslant \tau} \{\sigma(t, s_0)\}}\,.
	\]
Note that $s \in \omega(s_0)$ if, and only if there exists a sequence $(t_k)_{k \in \n}$ going to infinity in $\rplus$ such that
	\[
		\lim_{k \to \infty} \sigma(t_k, s_0) = s\,.
	\]

\begin{definition}[Bounded-Persistence]\label{def:boundedpersistence}
	A reaction network (\ref{eq:crntode}) is said to be {\em bound\-ed-persistent} if $\omega(s_0) \cap \partial\rplus^{n} = \varnothing$ for each $s_0 \gg 0$.
\mytriangle
\end{definition}

A {\em steady state} of a reaction network $G$ is any point $s_0 \geqslant 0$ such that $Nr(s_0) = 0$.

\begin{definition}[Boundary Steady State] 
A {\em boundary steady state} is any point $s_0 \in \partial \rplus^n$ such that $Nr(s_0) = 0$, in other words, any steady state that lies on the boundary. 
\mytriangle
\end{definition}

In the following proposition we collect  relationships among 
persistence, bounded-persistence, consistence, drainable siphons and the siphon/P-semiflow property. Details of the proof are given in Subsection~\ref{app:prop2} in the appendix.

\begin{proposition} \label{prop:persistent2bounded_persistent}
Consider a reaction network $G$.
\begin{enumerate}[{\em(}i\,{\em)}]
\item If $G$ is persistent, then it is bounded-per\-sist\-ent.
\item If $G$ is conservative and bounded-persistent, then it is persistent.
\item If $G$ is conservative and persistent, then it is consistent.	
\item If $G$ has no drainable siphons, then it is bounded-persistent.
\item If $G$ has the siphon/P-semiflow property, then the stoichiometric compatibility classes of $G$ that are not entirely contained in the boundary do not contain any boundary steady states.
\end{enumerate}
\end{proposition}

Conservative networks are a special case of dissipative networks (Definition \ref{def:dissipative}), for which bounded-persistence is also equivalent to persistence. These will be discussed in Subsection \ref{subsec:dissipative}.

\begin{remark}\label{rem:substitution}
In view of Corollary~\ref{cor:equivcrit}, if a reaction network has the siphon/P-semiflow property, then it has no drainable siphons and therefore it is bounded-persistent by Proposition~\ref{prop:persistent2bounded_persistent}({\em iv}\,).
\mybox
\end{remark}

The next example shows that not having any drainable siphons is not in general a necessary condition for the {bounded-}persistence of reaction networks.
\begin{example}[Lotka-Volterra Predator-Prey Model]\label{ex:lv}
	The Lotka-Volterra equations,
	\begin{equation}\label{eq:lv}
 \frac{dN}{dt} = N(t)(a - bP(t)) \qquad\qquad  \frac{dP}{dt} = P(t)(cN(t) - d)\,,	
	\end{equation}
	where $a, b, c, d$ are positive parameters, model the population sizes at time $t \geqslant 0$ of a predator species, $P(t)$, and its prey, $N(t)$, under the assumptions that $N(t)$ grows exponentially in the absence of predators, $P(t)$ decays exponentially in the absence of prey, and that both the growth rate of $P(t)$ and the depletion rate of $N(t)$ on account of predation are directly proportional to the population counts $N(t)$ and $P(t)$.
	
	Equations (\ref{eq:lv}) can be derived as (\ref{eq:crntode}) from the reaction network
	\begin{equation}\label{eq:lv_crnt}
		N \rightarrow 2N
		\quad\quad
		N + P \rightarrow P
		\quad\quad
		N + P \rightarrow N + 2P
		\quad\quad
		P \rightarrow 0\,,
	\end{equation}
	under mass-action kinetics (see, for instance, \cite{gunawardena-2003} for an account of mass-action kinetics). {Solutions of (\ref{eq:lv}) are known to be uniformly bounded away from zero. In fact, they are periodic \cite[Section 3.1]{murray-2002}. In particular, (\ref{eq:lv_crnt}) is bounded-persistent. However, the} minimal siphons of (\ref{eq:lv_crnt}) are $\{N\}$ and $\{P\}$, both of which are drainable on account of reactions $N + P \rightarrow P$ and $P \rightarrow 0$, respectively.
	\mydiamond
\end{example}

\begin{example}[Non-Drainable Siphons and Boundary Steady States]
The absence of drainable siphons does not in general preclude boundary steady states in stoichiometric compatibility classes that meet the interior of the positive orthant. For example, consider the reaction network with the reaction
$$ S \rightarrow 2S.$$
This reaction network has one stoichiometric compatibility class, namely $\rplus$, and a boundary steady state. But it has no drainable siphons.
\mydiamond
\end{example}

\subsection{Monomolecular Networks}

Iterating the simplification procedures discussed in this work will often result in what we shall refer to as monomolecular networks. Intuitively, these are reaction networks in which each reactant or product consists of at most a single species. The precise definition is given below in Definition \ref{def:monomolecular_definition}. 
For conservative monomolecular networks, the necessary and sufficient conditions for persistence given in Proposition \ref{prop:persistent2bounded_persistent}({\em iii}\,) and ({\em iv}\,) are actually equivalent, and characterized by the strong connectedness of the connected components of the network (Proposition \ref{prop:monomolecular_equivalence}).

\begin{definition}[Monomolecular Networks]\label{def:monomolecular_definition}
A reaction network $G = (\calS, \calC, \calR)$ is said to be {\em monomolecular} if, for each $y \in \calC$, either $y = 0$ or $y = S_{i}$ for some $i \in [n]$. In this case, we identify the nonzero {`}complexes{'} of $G$ with the corresponding {`}species.{'}
	\mytriangle
\end{definition}

\newcommand{\ilr}[1]{{\em (}#1\,{\em )}}
\newcommand{\olr}[1]{({\em #1}\,)}
\begin{proposition}\label{prop:monomolecular_equivalence}
Let $G = (\calS, \calC, \calR)$ be a monomolecular reaction network, and consider and the following seven properties.
\begin{enumerate}[{\em(}i\,{\em)}]
		\item $G$ is consistent.
		
		\item Each connected component of $G$ is strongly connected.
		
		\item $G$ has the siphon/P-semiflow property.
		
		\item $G$ has no drainable siphons.
		
		\item $G$ is {bounded}-persistent.
		
		\item $G$ is persistent.
\end{enumerate}
 Then the following implications hold:
\begin{center}
\ilr{i} $\Rightarrow$  \ilr{ii} $\Rightarrow$ \ilr{iii} $\Rightarrow$ \ilr{iv} $\Rightarrow$ \ilr{v} $\Leftarrow$ \ilr{vi}.
\end{center}
If the reaction network is conservative, then the six properties are equivalent.
\end{proposition}

\begin{proof}
Proposition~\ref{prop:persistent2bounded_persistent} and Corollary \ref{cor:equivcrit} guarantee that ({\em iii}\,) $\Rightarrow$ ({\em iv}\,)  $\Rightarrow$ ({\em v}\,) $\Leftarrow$ ({\em vi}\,) for any reaction network.
Furthermore, ({\em v}\,) $\Rightarrow$ ({\em vi}\,) $\Rightarrow$ ({\em i}\,) for conservative networks, also by Proposition~\ref{prop:persistent2bounded_persistent}.
	 
Thus, it  is sufficient to show that ({\em i}\,) $\Rightarrow$ ({\em ii}\,) and ({\em ii}\,) $\Rightarrow$  ({\em iii}\,) for arbitrary monomolecular networks.
		
	({\em i}\,) $\Rightarrow$ ({\em ii}\,).
	{Since $G$ is consistent by hypothesis, there exists a strictly positive T-semiflow $v\in \mathbb{R}^m_{>0}$, that is $Nv=0$. We prove below that $v$ is in the kernel of the incidence matrix of the reaction graph of $G$. Strong connectness of each connected component of $G$ then follows, for example, from \cite[Remark 6.1.1]{Feinbergss}.}
	
	{The incidence matrix $C_G$ of the reaction graph $(\calC,\calR)$ has $m$ columns and one row for each complex. The entries of the $j$-th column, corresponding to a reaction $R_{j} =  y \rightarrow y'$,  are all zero except for the entry corresponding to $y$, which is $-1$ and the entry corresponding to {$y'$}, which is $1$.   If $\calC=\{S_1,\dots,S_n\}$, then $C_G=N$ by definition, hence $C_G v=0$.
	If  $\calC=\{S_1,\dots,S_n,0\}$, then the first $n$ rows of $C_G$ agree  with $N$. Since {the sum of the rows of $C_G$ is zero}, we have
$$(C_G)_{n+1} \cdot v = - \sum_{i=1}^{n} (C_G)_i \cdot v  =- \sum_{i=1}^{n} N_i \cdot v =0.$$ 
We conclude once again that $C_G v =0$.}

({\em ii}\,) $\Rightarrow$  ({\em iii}\,). Let $(\calC_1, \calR_1), \ldots, (\calC_J, \calR_J)$ be the connected components of $(\calC, \calR)$, and {denote the canonical basis of $\r^n$ by $\{\cane_1, \ldots, \cane_n\}$}. Let $j \in [J]$. We have two possibilities.
	
	If $0 \notin \calC_j$, then any siphon of $G$ containing some species $S' \in \calC_j$ contains $\calC_j$. Indeed, for any other $S \in \calC_j$, there exists a reaction path connecting $S$ to $S'$. Thus, $S$ belongs to any siphon containing $S'$. Furthermore, 
	\[
		\sum_{i\colon S_i \in \calC_j} \cane_i
	\]
	is a P-semiflow of $G$. This follows from the fact that, for each reaction $S \rightarrow S' \in \calR_j$, the column of $N$ corresponding to $S \rightarrow S'$ has exactly two nonzero entries, namely, a $1$ in the row corresponding to $S'$, and a $-1$ in the row corresponding to $S$.
	
	If $0 \in \calC_j$, then by strong connectedness, {there is a reaction path from $0$ to any species $S \in \calC_j$}. By Example \ref{ex:inflows},  $S$ cannot belong to any siphon of $G$ and thus $\calC_j$ contains no siphons.
			We conclude that every siphon of $G$ contains the support of a P-semiflow.
\end{proof}

The property that every connected component of the reaction graph is strongly connected is also known in the literature as {\em weak reversibility} (see \cite[Definition 6.1]{gunawardena-2003}). Thus, Proposition \ref{prop:monomolecular_equivalence}, as well as other results further down, could well have been stated in these terms. In this work, the property of weak reversibility only comes up in the context of monomolecular networks. Thus, we chose to use the more informative, explicit description in terms of strong connectivity of the connected components.

\begin{example}[Persistence without Conservativity]
The hypothesis that the network is conservative in Pro\-po\-sition \ref{prop:monomolecular_equivalence}  is not superfluous for the full equivalence of the six statements. For example, the reaction network 
$0 \rightarrow A$ 	with, say, mass-action kinetics is persistent but not consistent and the implication	{({\em vii}\,) $\Rightarrow$ ({\em i}\,)}  does not hold.
\mydiamond
\end{example}

\section{Intermediates and Catalysts}\label{sec:operations}

In this section we define the concepts of intermediate and catalyst of a reaction network. We also describe the reaction networks that are obtained from their removal. After establishing these concepts and underlying terminology in Subsections \ref{subsec:intermediates} and \ref{subsec:catalysts}, we state our main results in Subsection \ref{subsec:mainresult}.

\subsection{Intermediates}\label{subsec:intermediates}

Consider a reaction network $G = (\calS, \calC, \calR)$. Let $\calY$ be a nonempty subset of $\calS$, and write
\[
	\calY = \{Y_1, \ldots, Y_p\}\,,
\qquad\text{and}\qquad
	\calS \backslash \calY = \{S_1, \ldots, S_q\}\,.
\]
Consider the following two properties.

\begin{itemize}
	\item[(I1)] Each complex $y \in \calC$ is either of the form $y = \alpha_1S_1 + \cdots + \alpha_qS_q$ for some $\alpha_1, \ldots, \alpha_q \geqslant 0$, or of the form $y = Y_i$ for some $i \in [p]$.
		In particular, we identify the `complexes' and `species' $Y_1, \ldots, Y_p$. (See also Definition \ref{def:monomolecular_definition}.)
	
	\item[(I2)] For each $Y \in \calY$, there exist $y, y' \in \calC \backslash \calY$, and reaction paths  from $y$ to $Y$ and from $Y$ to $y'$ such that all their non-endpoints are in $\calY$.
\end{itemize}

If (I1) and (I2) hold, then we may construct a reaction network $G^* = (\calS^*, \calC^*, \calR^*)$ as follows. We define $\calR^* := \calR^*_c \cup \calR^*_Y$, where $\calR^*_c \subseteq \calR$ is the set of reactions $y \rightarrow y' \in \calR$ such that $y, y' \in \calC \backslash \calY$, and $\calR^*_Y$ is constructed as the set of reactions $y \rightarrow y'$ such that $y, y' \in \calC \backslash \calY$, $y \neq y'$, and there is a reaction path in $G$ connecting $y$ to $y'$ such that all their non-endpoints are in $\calY$. 
We set $\calC^*$ to be the set of reactant and product complexes in the reactions in $\calR^*$, and we set $\calS^*$ to be the set of species that are part of some complex in $\calC^*$. {Note} that $\calS^*$ does not always coincide with $\calS \backslash \calY$, as illustrated in Example \ref{ex:ubiquitin1of3} below. {In the above description, we think of the reactant and product sides of a reaction $y \rightarrow y' \in \calR^*$ as the formal linear combinations of participating species alluded to  before.}

\begin{definition}[Intermediates]\label{def:intermediates}
	Let $G = (\calS, \calC, \calR)$ be a reaction network and $\calY$ be a nonempty subset of $\calS$. We call $\calY$ a {\em set of intermediate species of $G$}, if (I1) and (I2) hold. In this case, the reaction network $G^* = (\calS^*, \calC^*, \calR^*)$ defined as above is called the {\em reduction of $G$ by the removal of the set of intermediates $\calY$}. The elements of $\calY$ are then referred to as the {\em intermediate species of $G$}.
\mytriangle	
\end{definition}

For brevity, we will often write simply intermediates instead of intermediate species.

\begin{example}[A Ubiquitination Model]\label{ex:ubiquitin1of3}
	Consider the following reaction network model for Ring1B/Bmi1 ubiquitination \cite{nguyen--munoz-garcia--maccario--ciechanover--kolch--kholodenko-2011}.

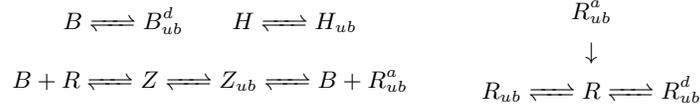
\begin{center}%
	\hfill%
	\begin{minipage}[c]{0.5\textwidth}
		\[
			B \longleftrightarrow B_{ub}^d 
			\quad\quad
			H \longleftrightarrow H_{ub}
		\]
		\vspace{-1ex}
		\[
			B + R \longleftrightarrow Z \longleftrightarrow Z_{ub} \longleftrightarrow B + R_{ub}^a
		\]
	\end{minipage}%
	\hfill\null%
	\begin{minipage}[c]{0.3\textwidth}
		\[
			R_{ub}^a
		\]
		\[
			\downarrow
		\]
		\[
			R_{ub} \longleftrightarrow R \longleftrightarrow R_{ub}^d
		\]
	\end{minipage}%
	\hfill\null%
\end{center}
Note that
	\[
		\calY := \{B_{ub}^d, H, R_{ub}, R_{ub}^d, Z, Z_{ub}\}
	\]
	is a set of intermediate species of the network. This network can be reduced to
	\[
		B + R \longleftrightarrow B + R_{ub}^a
		\quad\quad
		R_{ub}^a \longrightarrow R
	\]
	by removing these intermediates and collapsing the paths in which they appear, as described above.
	
	We emphasize that $\calS^*$ does not always coincide with $\calS \backslash \calY$. In this example, $H_{ub}$ is in $\calS \backslash \calY$, but not in $\calS^*$. We also note that the same network $G^*$ may arise from removing a different set of intermediates. For instance, in this example, we could have set $H_{ub}$ as an intermediate in place of $H$.
\mydiamond
\end{example}

\subsubsection*{Removing One Intermediate At A Time}

In the proofs of some of our results concerning the removal of intermediates, we use induction on the number of intermediates removed. Thus, a discussion of how the intermediates in a set of intermediates may be iteratively removed, one at a time, is warranted.

Let $G = (\calS, \calC, \calR)$ be a reaction network, and suppose $\calY = \{Y_1, \ldots, Y_p\}$ is a set of intermediates of $G$. Set $G_p := G$. It follows directly from the definition that any nonempty subset of $\calY$ is a set of intermediates of $G$. In particular, $\{Y_p\}$ is a set of intermediates of $G_p$. Let $G_{p-1}$ be the reduction of $G_p$ by the removal of the set of intermediates $\{Y_p\}$. Now $\{Y_1, \ldots, Y_{p-1}\}$ is a set of intermediates of $G_{p-1}$. In particular, $\{Y_{p-1}\}$ is a set of intermediates of $G_{p-1}$. We define $G_{p-2}$ to be the reduction of $G_{p-1}$ by the removal of the set of intermediates $\{Y_{p-1}\}$. Iterating this process $p$ times, we obtain a sequence $G_p, \ldots, G_1, G_0$ such that $G_p = G$, and $G_{i-1}$ is the reduction of $G_i$ by the removal of the set of intermediates $\{Y_i\}$, $i = p, \ldots, 1$.

\begin{lemma}\label{lemm:intermediateinduction}
If $G$, $\calY$, and $G_p, \ldots, G_1, G_0$ are like in the above construction, and $G^*$ is the reduction of $G$ by the removal of the set of intermediates $\calY$, then $G_0 = G^*$.
\end{lemma}

\begin{proof}
	We use induction on $p$. The claim is trivial for $p = 1$. So, suppose it has been proven to be true for the removal of up to $p$ intermediates, for some $p \geqslant 1$. Let $\calY = \{Y_1, \ldots, Y_p, Y_{p+1}\}$ be a set of intermediates of $G$. As noted above, $\{Y_2, \ldots, Y_{p+1}\}$ is a set of intermediates of $G$. Let $G_1^*$ be the reduction of $G$ obtained by their removal. By the induction hypothesis, $G_1^* = G_1$, and so $\calR^*_1 = \calR_1$. We want to show that $\calR_0 = \calR^*$.
	
	$\calR^* \subseteq \calR_0$. Let $y \rightarrow y'$ be any reaction in $\calR^*$. If $y \rightarrow y' \in \calR$, then $y \rightarrow y' \in \calR^*_1$, and so $y \rightarrow y' \in \calR_0$. So, suppose $y \rightarrow y' \notin \calR$. Then there exist $Y^{(1)}, \ldots Y^{(\ell)} \in \calY$ such that
	\[
		y \longrightarrow Y^{(1)} \longrightarrow \cdots \longrightarrow Y^{(\ell)} \longrightarrow y'
	\]
	is a reaction path in $G$. 
	If $Y^{(1)}, \ldots Y^{(\ell)} \in \{Y_2, \ldots, Y_{p+1}\}$, then $y \rightarrow y' \in \calR_1^*$, and so $y \rightarrow y' \in \calR_0$ like in the previous case. Otherwise, we have $Y_1 = Y^{(i)}$ for some $i \in [\ell]$. But now
	\begin{equation}\label{eq:y-Y_1-yprime}
		y \longrightarrow Y_1 \longrightarrow y'
	\end{equation}
	is a reaction path in $G_1^*$, and so $y \rightarrow y' \in \calR_0$ once again.
	
	$\calR_0 \subseteq \calR^*$. Let $y \rightarrow y'$ be any reaction in $\calR_0$. If $y \rightarrow y' \in \calR_1^*$, then there exists a reaction path connecting $y$ to $y'$ in $G$ such that all its non-endpoints belong to $\{Y_2, \ldots, Y_{p+1}\}$. In this case, $y \rightarrow y' \in \calR^*$. If $y \rightarrow y' \notin \calR_1^*$, then (\ref{eq:y-Y_1-yprime}) is a reaction path in $G_1^*$. In this case there are reaction paths in $G$ connecting $y$ to $Y_1$ and $Y_1$ to $y'$, all non-endpoints of which belong to $\{Y_2, \ldots, Y_{p+1}\}$. Concatenating these two reaction paths we obtain a reaction path in $G$ connecting $y$ and $y'$ such that all its non-endpoints belong to $\calY$. It follows once again that $y \rightarrow y' \in \calR^*$.
\end{proof}

\subsection{Catalysts}\label{subsec:catalysts}

Consider a reaction network $G = (\calS, \calC, \calR)$. Let $\calE$ be a nonempty subset of $\calS$, and write
\[
	\calE = \{E_1, \ldots, E_p\}\,, \qquad \text{and}\qquad 
	\calS \backslash \calE = \{S_1, \ldots, S_q\}\,.
\]
Consider the following two properties.
\begin{itemize}
	\item[(C1)] For each reaction
		\[
			\sum_{i = 1}^{q} \alpha_iS_i + \sum_{i = 1}^{p} \beta_iE_{i} \longrightarrow \sum_{i = 1}^{q} \alpha'_iS_i + \sum_{i = 1}^{p} \beta'_iE_{i}
		\]
		in $\calR$, we have
		\[
			\sum_{i = 1}^{p} \beta_iE_{i} = \sum_{i = 1}^{p} \beta'_iE_{i}
		\qquad \text{or}\qquad
			\alpha_1 = \alpha'_1 = \cdots = \alpha_q = \alpha'_q = 0\,.
		\]
		
	\item[(C2)] The subnetwork $G_\calE = (\calS_\calE, \calC_\calE, \calR_\calE)$ implied by $\calE$ (refer to Definition \ref{def:implied_network}) has no drainable or self-replicable siphons (equivalently, has the siphon/P-semiflow property).
\end{itemize}
If (C1) and (C2) hold, then we may construct a reaction network $G^* = (\calS^*, \calC^*, \calR^*)$ as follows. We set $\calR^*$ to be the set of reactions
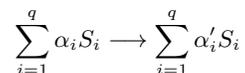
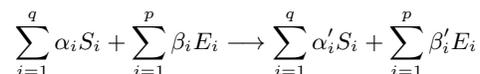
\[
	\sum_{i = 1}^{q} \alpha_iS_i \longrightarrow \sum_{i = 1}^{q} \alpha'_iS_i
\]
such that
\[
	\sum_{i = 1}^{q} \alpha_iS_i + \sum_{i = 1}^{p} \beta_iE_{i} \longrightarrow \sum_{i = 1}^{q} \alpha'_iS_i + \sum_{i = 1}^{p} \beta'_iE_{i}
\]
belongs to $\calR$, and $\alpha_{i_0} > 0$ or $\alpha'_{i_0} > 0$ for some $i_0 \in [q]$. We then set $\calC^*$ to be the set of reactants and products in these reactions, and set $\calS^*$ to be the set of species that are part of some complex in $\calC^*$. Contrary to what happened with intermediates, $\calS^*$ always agrees with $\calS \backslash \calE$.

\begin{definition}[Catalysts]\label{def:catalysts}
	Let $G = (\calS, \calC, \calR)$ be a reaction network and $\calE$ be a nonempty subset of $\calS$. We call $\calE$ a {\em set of catalysts of $G$} if (C1) and (C2) hold. In this case, the reaction network $G^* = (\calS^*, \calC^*, \calR^*)$ defined as above is called a {\em reduction of $G$ by the removal of the set of catalysts $\calE$}. The elements of $\calE$ are then referred to as the {\em catalysts of $G$}.
\mytriangle	
\end{definition}

Typically (C2) is checked via Proposition \ref{prop:monomolecular_equivalence}, by showing that $G_\calE$ is a mon\-o\-mo\-lec\-u\-lar network and each connected component of its reaction graph is strongly connected, as we shall see in some of the examples in {the next section}. However, the theory allows for catalysts to interact in more complex, yet still biologically meaningful ways, for instance, in reversible reactions of the forms
	\[
		E_1 + E_2 \longleftrightarrow 2E_3\,,
		\quad\quad
		E_1 + E_2 \longleftrightarrow E_3 + E_4\,,
		\quad
		\text{or}
		\quad
		2E_1 \longleftrightarrow E_2\,.
	\]

\begin{example}[A Ubiquitination Model (Continued)]\label{ex:ubiquitin2of3}
	Consider the network 
	\[
		B + R \longleftrightarrow B + R_{ub}^a
		\quad\quad
		R_{ub}^a \longrightarrow R\,,
	\]
	obtained from the ubiquitination model in Example \ref{ex:ubiquitin1of3} after intermediates were removed. Note that $\calE := \{B\}$	is a set of catalysts. Thus, this network can be further reduced to
	\[
		R \longleftrightarrow R_{ub}^a
	\]
	by removing $B$ and projecting the reactions as described above.
	
	Note that $B$ is not a catalyst of the original ubiquitination model in Example \ref{ex:ubiquitin1of3}. In realistic biochemical models, it is often the case that catalysts in the sense of Definition \ref{def:catalysts} only emerge after some intermediates are removed.
\mydiamond	
\end{example}

\subsection{Main Results}\label{subsec:mainresult}

We are now ready to precisely restate Theorems \ref{thm:informal_main}, \ref{thm:informal_main2} and \ref{thm:informal_main3} in the introduction, and consider a few examples. The proofs will be given in Section \ref{sec:proofofmainresult}.

\theoremstyle{plain}
\newtheorem{realtheorem}{Theorem}

\begin{realtheorem}[Removal of Intermediates]\label{thm:main}
	Suppose a reaction network $G^*$ is obtained from a reaction network $G$ by the removal of a set of intermediates. Then, 
		\begin{itemize}
			\item[{\em(}i\,{\em)}] $G$ has no drainable (respectively, self-replicable) siphons if, and only if $G^*$ has no drainable (respectively, self-replicable) siphons;
			\item[{{\em(}ii\,{\em)}}] $G$ has the siphon/P-semiflow property if, and only if $G^*$ has the siphon/P-semiflow property;
			\item[{{\em(}iii\,{\em)}}] $G$ is consistent if, and only if $G^*$ is consistent; and
			\item[{{\em(}iv\,{\em)}}] if $G$ is conservative, then $G^*$ is conservative; conversely, if $G^*$ is conservative and $0 \notin \calC$, then $G$ is also conservative.
	\end{itemize} 
\end{realtheorem}

\begin{realtheorem}[Removal of Catalysts]\label{thm:main2}
	Suppose a reaction network $G^*$ is obtained from a reaction network $G$ by the removal of a set of catalysts $\calE$. Then, 
		\begin{itemize}
			\item[{\em(}i\,{\em)}] $G$ has no drainable (respectively, self-replicable) siphons if, and only if $G^*$ has no drainable (respectively, self-replicable) siphons;
			\item[{{\em(}ii\,{\em)}}] $G$ has the siphon/P-semiflow property if, and only if $G^*$ has the siphon/P-semiflow property;
			\item[{{\em(}iii\,{\em)}}] if $G$ is consistent, then $G^*$ is consistent; conversely, if $G^*$ is consistent and $G_\calE$ is conservative, then $G$ is consistent; and
			\item[{{\em(}iv\,{\em)}}] if $G$ is conservative, then $G^*$ is conservative; conversely, if $G^*$ is conservative and $G_\calE$ is conservative, then $G$ is conservative.
		\end{itemize} 
\end{realtheorem}

{Combining Theorems~\ref{thm:main} and \ref{thm:main2} with Proposition \ref{prop:monomolecular_equivalence} 
we obtain the following corollary.}

\begin{corollary}\label{cor:intcat}
Suppose a monomolecular reaction network $G^*$ is obtained by iteratively removing sets of intermediates and catalysts from a reaction network $G$. If each of the connected components of $G^*$ is strongly connected, then $G$ is bounded-persistent and has no boundary steady states in any stoichiometric compatibility class that is not already contained in the boundary of the positive orthant. Furthermore, if $G$ is conservative, then $G$ is persistent and consistent.  
\end{corollary}

\newcommand{\primitive}{primitive}
\newcommand{\Primitive}{Primitive}

\begin{definition}[Primitive Networks]
	A reaction network $G = (\calS, \calC, \calR)$ is said to be {\em \primitive} ({\em with respect to the removal of catalysts or intermediates}\,) if no subset of $\calS$ is a set of catalysts or intermediates of $G$. If iteratively removing sets of intermediates and catalysts of a reaction network $G$ results in a \primitive\ reaction network $G^*$, then we refer to $G^*$ as a {\em \primitive\ reduction of $G$}.
	\mytriangle
\end{definition}

\begin{realtheorem}[Uniqueness of The \Primitive\ Reduction]\label{thm:primitive}
	Let $G$ be a reaction network, and suppose $G^*_1$ and $G^*_2$ are \primitive\ reductions of $G$. Then $G^*_1 = G^*_2$.
\end{realtheorem}

Observe that Theorem \ref{thm:primitive} is more than just a theoretical curiosity. As noted in Example \ref{ex:ubiquitin1of3}, choosing a set of intermediates or catalysts to remove is not something that can always be done in a unique way at each stage of the simplification process. Thus, knowing that one would always obtain the same minimally simplified reaction network regardless of the order in which catalysts and intermediates are removed has also practical relevance.

\begin{example}[A Ubiquitination Model (Concluded)]\label{ex:ubiquitin3of3}
	The network
	\[
		R \longleftrightarrow R_{ub}^a
	\]
	is a strongly connected monomolecular network. 
	{By Corollary~\ref{cor:intcat}, so long as the reaction rates of  the ubiquitination model from Example \ref{ex:ubiquitin1of3} satisfy our hypotheses, we conclude that the network is persistent.}
\mydiamond
\end{example}

We emphasize that the procedures of removal of intermediates and catalysts carried out in Examples \ref{ex:ubiquitin1of3} and \ref{ex:ubiquitin2of3}, as well as the analysis of the emerging underlying substrate network for strong connectedness in Example \ref{ex:ubiquitin3of3}, are essentially graphical. More specifically, one need not do any calculations with the stoichiometric matrix {or the reaction rates}.

In Theorem \ref{thm:main2}({\em iii}\,), the hypothesis that $G_\calE$ be conservative is not superfluous. If that is not the case, then it might happen that $G^*$ is consistent and $G$ is not, as shown in Example \ref{ex:nonconservativeGE} below. However, if $G$ is consistent, then $G^*$ is consistent regardless of whether $G_\calE$ is conservative or not, as shown {later} in Lemma \ref{lemm:consistencypreserved}.

\begin{example}[Non-Conservative $G_\calE$]\label{ex:nonconservativeGE}
Consider the reaction network
	\[
		G:
		\quad
		A + E \longleftrightarrow B + E
		\quad\quad
		0 \longrightarrow E\,.
	\]
	The singleton $\calE := \{E\}$ is a set of catalysts of $G$, the removal of which yields the reaction network
	\[
		G^*:
		\quad
		A \longleftrightarrow B\,.
	\]
	{By Proposition \ref{prop:monomolecular_equivalence}, $G^*$ is consistent. } 	The stoichiometric matrix of $G$ is
	\[
		N = 
		\left[
			\begin{array}{rrr}
				-1 & 1 & 0 \cr
				1 & -1 & 0 \cr
				0 & 0 & 1
			\end{array}
		\right]\,.
	\]
	Any T-semiflow of $G$ must have its third coordinate equal to zero, so $G$ is not consistent.
\mydiamond	
\end{example}

\section{Examples}\label{sec:examples}

We shall apply Theorems \ref{thm:main} and \ref{thm:main2} to two main classes of reaction networks. In Subsection \ref{subsec:ptm}, we give necessary and sufficient conditions for cascades of a class of post-translational modification (PTM) systems to be persistent. The reaction network {\eqref{eq:prototype}} in the introduction, as well as the ubiquitination model discussed in Examples \ref{ex:ubiquitin1of3}, \ref{ex:ubiquitin2of3} and \ref{ex:ubiquitin3of3}, will turn out to be special cases of PTM systems. In Subsection \ref{subsec:dissipative}, we argue that a nonconservative reaction network may still be shown to be persistent when it has no drainable siphons as long as it can be {also} shown to be dissipative. Finally, in Subsection \ref{subsec:bss}, {we apply our results to a model of Wnt signaling that focuses on shuttling and degradation  \cite{shiu--sturmfels-2010}.}

\subsection{Cascades of PTM Systems}\label{subsec:ptm}

In this subsection, we study the persistence of {cascades of} a class {of} PTM systems. Combining Theorems \ref{thm:main} and \ref{thm:main2} with {Propositions~\ref{prop:persistent2bounded_persistent},  \ref{prop:monomolecular_equivalence} and Corollary \ref{cor:equivcrit}, }
we will achieve necessary and sufficient conditions for persistence of cascades of PTM systems in terms of strong connectedness of the connected components of the underlying substrate network of each layer.

\subsubsection{PTM Systems}

\newcommand{\Enz}{{\rm Enz}}
\newcommand{\Sub}{{\rm Sub}}
\newcommand{\Int}{{\rm Int}}

Consider a reaction network $G = (\calS, \calC, \calR)$. Let
\[
	\calS = \Enz \cup \Sub \cup \Int
\]
be a partition of the species set. Thus, $\Enz$, $\Sub$, and $\Int$ are pairwise disjoint. Consider the following properties.
\begin{itemize}	
	\item[(M1)] The reactions set $\calR$ can be partitioned into a disjoint union of subsets
		\[
			\calR = \calR_{\Sub} \sqcup \calR_{\Sub + \Enz} \sqcup \calR_{\rightarrow \Int} \sqcup \calR_{\Int \rightarrow} \sqcup \calR_{\Int}\,,
		\]
		which are uniquely determined from the partition $\calS = \Enz \cup \Sub \cup \Int$ by the inclusions
		\[
			\begin{aligned}
				\calR_{\Sub} &\subseteq \{S \rightarrow S'\,|\ S, S' \in \Sub\}\,, \\[1ex]
				\calR_{\Sub + \Enz} &\subseteq \{S + E \rightarrow S' + E\,|\ E \in \Enz\,,\ \text{and}\ S, S' \in \Sub\}\,, \\[1ex]
				\calR_{\rightarrow \Int} &\subseteq \{S + E \rightarrow Y'\,|\ E \in \Enz\,,\ S \in \Sub\,,\ \text{and}\ Y' \in \Int\}\,, \\[1ex]
				\calR_{\Int \rightarrow} &\subseteq \{Y \rightarrow S' + E\,|\ E \in \Enz\,,\ S' \in \Sub\,,\ \text{and}\ Y \in \Int\}\,, \\[1ex]
				\calR_{\Int} &\subseteq \{Y \rightarrow Y'\,|\ Y, Y' \in \Int\}\,.
			\end{aligned}
		\]
		
	\item[(M2)] $\Int$ is either empty or a set of intermediates of $G$.
	
	\item[(M3)] If
		\[
			S + E \longrightarrow Y^{(1)} \longrightarrow \cdots \longrightarrow Y^{(\ell)} \longrightarrow S' + E'
		\]
		is a reaction path in $G$ for some $E, E' \in \Enz$, some $S, S' \in \Sub$, and some $Y^{(1)}, \ldots, Y^{(\ell)} \in \Int$, then $E = E'$.
\end{itemize}

\begin{definition}[PTM Systems]
	Let $G = (\calS, \calC, \calR)$ be a reaction network, and let
	\[
		\calS = \Enz \cup \Sub \cup \Int
	\]
	be a partition of the species set. We say that $G$ is a {\em PTM system} with {\em enzyme} set $\Enz$, {\em substrate} set $\Sub$, and {\em intermediates} set $\Int$ if it has properties (M1)--(M3) above.
\mytriangle	
\end{definition}

Let $G = (\calS, \calC, \calR)$ be a PTM system. If $\Int = \varnothing$, then set $G^* = (\calS^*, \calC^*, \calR^*) := G$. Otherwise, let $G^* = (\calS^*, \calC^*, \calR^*)$ be the network obtained from $G$ by the removal of the set of intermediates $\Int$. Thus,
	\[
		\calR^* = \calR_{\Sub} \cup \calR_{\Sub + \Enz} \cup \calR^Y_{\Sub + \Enz}\,,
	\]
	where $\calR^Y_{\Sub + \Enz}$ is the set of reactions of the form $S + E \rightarrow S' + E$ such that
	\[
		S + E \longrightarrow Y^{(1)} \longrightarrow \cdots \longrightarrow Y^{(\ell)} \longrightarrow S' + E
	\]
	is a reaction path in $G$ for some $E \in \Enz$, some $S, S' \in \Sub$ such that $S \neq S'$, and some $Y^{(1)}, \ldots, Y^{(\ell)} \in \Int$. 
	
Now $\calS^* \subseteq \Enz \cup \Sub$, and $\Enz^* := \Enz \cap \calS^*$, if nonempty, is a set of catalysts of $G^*$. Indeed, it follows directly from the form of the reactions that (C1) holds, and the subnetwork of $G^*$ implied by $\Enz^*$ is the empty network, so (C2) also holds. If $\Enz^* = \varnothing$, then we set $G^{**} = (\calS^{**}, \calC^{**}, \calR^{**}) := G^*$. Otherwise, let $G^{**} = (\calS^{**}, \calC^{**}, \calR^{**})$ be the network obtained from $G^*$ by the removal of the set of catalysts $\Enz^*$. Then $G^{**}$ is a monomolecular network consisting of the reactions $S \rightarrow S'$ such that $S + \alpha E \rightarrow S' + \alpha E \in \calR^*$ for some $E \in \Enz^*$, some $\alpha \in \{0, 1\}$, and some $S, S' \in \Sub$ such that $S \neq S'$. We refer to $G^{**}$ as the {\em underlying substrate network of $G$}. 

Regardless of whether $\Int$ or $\Enz$ are empty or nonempty, we shall abuse the terminology and refer to the reaction network $G^*$ as the reaction network obtained from $G$ by the removal of the set of intermediates $\Int$ and to $G^{**}$ as the reaction network obtained from $G^*$ by the removal of the set of catalysts $\Enz^*$, for simplicity.

{Note that $G^*$  and $G^{**}$ are themselves PTM systems with an empty set of intermediates and empty sets of intermediates and catalysts respectively.}

{By \cite[Equations (16) and (17)]{thomson--gunawardena-2009}, any PTM system is conservative (see also Lemma \ref{lemm:conservativecascades} below). In particular, persistence and bounded-persistence are equivalent for PTM systems in view of Proposition \ref{prop:persistent2bounded_persistent}.}

\begin{proposition}\label{prop:ptmpersistence}
	Let $G$ be a PTM system. Then the following properties are equivalent.
	\begin{itemize}
		\item[{{\em(}i\,{\em)}}] $G$ is consistent.
		
		\item[{{\em(}ii\,{\em)}}] Each connected component of the underlying substrate network $G^{**}$ is strongly connected.
		
		\item[{{\em(}iii\,{\em)}}] $G$ has the siphon/P-semiflow property.
		
		\item[{\em(}iv\,{\em)}] $G$ has no drainable siphons.
		
		\item[{{\em(}v\,{\em)}}] $G$ is persistent.
	\end{itemize}
\end{proposition}

\begin{proof}
Using that $G$ is conservative, Proposition~\ref{prop:persistent2bounded_persistent} and Corollary \ref{cor:equivcrit} 
give the following implications:
({\em iii}\,) $\Rightarrow$   ({\em iv}\,)  $\Rightarrow$ ({\em v}\,)   $\Rightarrow$  ({\em i}\,).
Thus, it  is remains to show that ({\em i}\,) $\Rightarrow$ ({\em ii}\,) $\Rightarrow$ ({\em iii}\,).

	{({\em i}\,)} $\Rightarrow$ {({\em ii}\,)}. It follows from Theorems \ref{thm:main}({\em iii}\,) and \ref{thm:main2}({\em iii}\,) that $G^{**}$ is consistent. Since $G^{**}$ is conservative, Proposition \ref{prop:monomolecular_equivalence} gives that each connected component of $G^{**}$ is strongly connected.

{({\em ii}\,)} $\Rightarrow$ {({\em iii}\,)}. By Proposition \ref{prop:monomolecular_equivalence}, $G^{**}$ has the siphon/P-semiflow property. It then follows by Theorems \ref{thm:main2}({\em ii}\,) and \ref{thm:main}({\em ii}\,), respectively, that $G^*$ and, consequently, $G$ have the siphon/P-semiflow property.
\end{proof}

\begin{remark}
	In view of Proposition \ref{prop:monomolecular_equivalence}, statement {({\em ii}\,)} in Proposition \ref{prop:ptmpersistence} is equivalent to each of the statements that the underlying substrate network $G^{**}$ of $G$ is consistent, has the siphon/P-semiflow property, has no drainable siphons, or is persistent. Thus, either of these properties could also be checked to establish the persistence of $G$.
\mybox
\end{remark}

\begin{example}[An $n$-Site Phosphorylation Mechanism]\label{ex:2phosphorylation}
	The sequential and distributive $n$-site phosphorylation mechanism given by
	
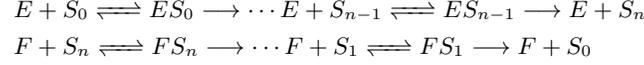
\[
		\begin{aligned}
			& E + S_0 \longleftrightarrow ES_0 \longrightarrow \cdots E + S_{n-1} \longleftrightarrow ES_{n-1} \longrightarrow E + S_n \\
			& F + S_n \longleftrightarrow FS_n \longrightarrow \cdots F + S_1 \longleftrightarrow FS_1 \longrightarrow F + S_0
		\end{aligned}
	\]
	is a PTM system with $\Int = \{ES_0, ES_1, \ldots, ES_{n-1}, FS_n, FS_{n-1}, \ldots, FS_1\}$, 
$\Enz = \{E, F\}$, and $\Sub = \{S_0, S_1, \ldots, S_n\}$.
	The underlying substrate network is
	\[
		S_0 \longleftrightarrow S_1 \longleftrightarrow \cdots \longleftrightarrow S_n\,.
	\]
	It consists of a single strongly connected component, so the PTM system is persistent by Proposition \ref{prop:ptmpersistence}.
\mydiamond
\end{example}

{%
\begin{example}
Consider the PTM system
	\[
	 E + S_0 \longleftrightarrow ES_0 \longrightarrow E + S_1. 
	\]
The underlying substrate network, $S_0 \longrightarrow S_1$,  is not strongly connected.
We conclude that the PTM system is not persistent.
\mydiamond
\end{example}
}

\subsubsection{Signaling Cascades of PTM Systems}

We now discuss a formalism for cascades of PTM systems. Intuitively, a signaling cascade of PTM systems is a reaction network that can be decomposed into a hierarchy of PTM systems in such a way that substrates at a certain level, or layer, may act as {enzymes} in lower levels (but not in higher levels).

Consider a reaction network $G = (\calS, \calC, \calR)$, and write the species, complex and reaction sets of the network as (not necessarily disjoint) unions,
\begin{equation}\label{eq:cascades}
		\calS = \bigcup_{i = 1}^T \calS_i\,,
		\quad
		\calC = \bigcup_{i = 1}^T \calC_i\,,
		\quad
		\text{and}
		\quad
		\calR = \bigcup_{i = 1}^T \calR_i\,.
\end{equation}
	Consider the following properties.
\begin{itemize}
	\item[(F1)] For each $i \in [T]$, $G_i := (\calS_i, \calC_i, \calR_i)$ is a PTM system with enzyme, substrate, and intermediates  sets, respectively, $\Enz_i$, $\Sub_i$, and $\Int_i$.
		
	\item[(F2)] $\displaystyle \Sub_j \cap \left( \bigcup_{i = 1}^{j-1} \Sub_i \right) = \varnothing$, $j = 2, \ldots, T$.
	
	\item[(F3)] $\displaystyle \Enz_j \cap \left( \bigcup_{i = 1}^j \Sub_i \right) = \varnothing$, $j = 1, \ldots, T$.
	
	\item[(F4)] $\displaystyle \left( \bigcup_{i = 1}^T \Int_i \right) \cap \left( \bigcup_{i = 1}^T (\Enz_i \cup \Sub_i) \right) = \varnothing$.
\end{itemize}

\begin{definition}[Signaling Cascades of PTM Systems]\label{def:ptmcascades}
	Let $G = (\calS, \calC, \calR)$ be a reaction network. We say that $G$ is a {\em signaling cascade of PTM systems} if {there is a decomposition of the species, complex and reaction sets as in \eqref{eq:cascades} that satisfy} properties (F1)--(F4). In this case, we set
	\[
		\Enz := \bigcup_{i = 1}^T \Enz_i\,,
		\quad
		\Sub := \bigcup_{i = 1}^T \Sub_i\,,
		\quad
		\text{and}
		\quad
		\Int := \bigcup_{i = 1}^T \Int_i\,,
	\]
	and the PTM systems $G_1 = (\calS_1, \calC_1, \calR_1), \ldots,$ $G_T = (\calS_T, \calC_T, \calR_T)$ are referred to as the {\em layers} of the cascade.
	\mytriangle	
\end{definition}

\begin{remark}\label{rk:ptmcascade}
If $G = (\calS, \calC, \calR)$ is a signaling cascade of PTM systems, then by (F1)
and (F4), the species set $\calS$ can be partitioned into the two subsets $\Enz \cup \Sub$ and $\Int$. That is, 
$\calS=(\Enz \cup \Sub) \sqcup \Int$. Furthermore, $\Int$ is a set of intermediates of $G$, provided that it is nonempty.
\mybox 
\end{remark}

Observe that (F3) implies that any enzyme that is a substrate in some layer may appear in any layer below it, and not just the one immediately below the layer where it acts as a substrate. Thus, the layer hierarchy implied in the definition of signaling cascades of PTM systems may be a tree, in other words, it is not constrained to linear, sequential relationships where each layer can only provide the layer immediately after with enzymes.

Signaling cascades of PTM systems are {always} conservative.

\begin{lemma}\label{lemm:conservativecascades}
	Any signaling cascade of PTM systems is conservative.	
\end{lemma}

\begin{proof}
	Let $G = (\calS, \calC, \calR)$ be a cascade of PTM systems. Write $\calS = \{S_1, \ldots, S_n\}\,.$ {With the notation in Definition~\ref{def:ptmcascades} and using Remark~\ref{rk:ptmcascade}}, for each $i \in [n]$, set
	\[	
		\omega_i :=
		\left\{
		\begin{array}{rl}
			1\,, &\text{if}\ S_i \in \Enz \cup \Sub \\[1ex]
			2\,, &\text{if}\ S_i \in \Int\,.
		\end{array}
		\right.
	\]
	Then $\omega := (\omega_1, \ldots, \omega_n)$ is a conservation law of $G$. This can be readily seen from the possible forms a reaction in $\calR$ may take. Since every entry of $\omega$ is strictly positive, this means $G$ is conservative.
\end{proof}

\begin{proposition}\label{prop:cascades}
	Let $G$ be a signaling cascade of PTM systems. Then the following properties are equivalent.
	\begin{itemize}
		\item[{{\em(}i\,{\em)}}] $G$ is consistent.
		
		\item[{{\em(}ii\,{\em)}}] The connected components of the underlying substrate network of each layer of $G$ are strongly connected.
		
		\item[{{\em(}iii\,{\em)}}] $G$ has the siphon/P-semiflow property.
		
		\item[{\em(}iv\,{\em)}] $G$ has no drainable siphons.
		
		\item[{\em(}v\,{\em)}] $G$ is persistent.
	\end{itemize}
\end{proposition}

The proof of Proposition \ref{prop:cascades} will be given in the next subsubsection.

\begin{example}[Double Phosphorylation Cascade]
	Consider the concatenation of double phosphorylation mechanisms from Example \ref{ex:2phosphorylation} given by the reaction network
	\[
		\arraycolsep=0.25em
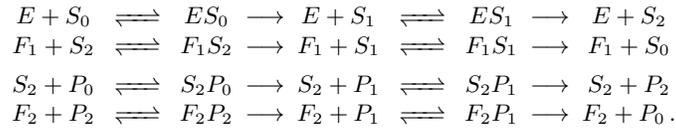
		\begin{array}{ccccccccc}
			E + S_0 & \longleftrightarrow & ES_0 & \longrightarrow & E + S_1 & \longleftrightarrow & ES_1 & \longrightarrow & E + S_2 \\
			F_1 + S_2 & \longleftrightarrow & F_1S_2 & \longrightarrow & F_1 + S_1 & \longleftrightarrow & F_1S_1 & \longrightarrow & F_1 + S_0 \\[1ex]
		
			S_2 + P_0 & \longleftrightarrow & S_2P_0 & \longrightarrow & S_2 + P_1 & \longleftrightarrow & S_2P_1 & \longrightarrow & S_2 + P_2 \\
			F_2 + P_2 & \longleftrightarrow & F_2P_2 & \longrightarrow & F_2 + P_1 & \longleftrightarrow & F_2P_1 & \longrightarrow & F_2 + P_0\,.
		\end{array}
	\]
	The double phosphorylation of a substrate $S_0$ is catalyzed by a kinase $E$, and the dephosphorylation of its singly and doubly phosphorylated forms is catalyzed by a phosphatase $F_1$. The doubly phosphorylated form $S_2$ of $S_0$ then acts as a kinase in a similar double phosphorylation/dephosphorilation mechanism for another substrate $P_0$. This is a signaling cascade of PTM systems with
\begin{align*}
		\Enz_1 & = \{S_2, F_2\}\,, \\ \Sub_1 & = \{P_0, P_1, P_2\}\,, \\  \Int_1 &= \{S_2P_0, S_2P_1, F_2P_2, F_2P_1\}\,, \\
		\Enz_2 & = \{E, F_1\}\,, \\	\Sub_2 & = \{S_0, S_1, S_2\}\,,	 \\   \Int_2 &= \{ES_0, ES_1, F_1S_2, F_1S_1\}\,.
\end{align*}
	Each of the layers of the cascade coincides with the double phosphorylation mechanism in Example \ref{ex:2phosphorylation} {with $n=2$}. In particular, {Proposition \ref{prop:cascades}\olr{ii} holds, hence the network is {persistent}.}
\mydiamond
\end{example}

In \cite{gnacadja-2011}, the persistence of a class of cascades of PTM systems (there called {\em cascaded binary enzymatic networks}) is studied under mass-action kinetics. There is an overlap between the class of networks studied in \cite{gnacadja-2011} and the class of cascades of PTM systems considered here, although neither is more general than the other, nor do they agree. For instance, we allow for individual enzymes to take part in reactions in more than one layer of the cascade. In \cite{gnacadja-2011}, sufficient conditions for a stronger concept of persistence ({\em vacuous persistence}) are given in terms of the so-called {\em futility} of the network \cite[Theorem 6.7]{gnacadja-2011}. 
{For conservative networks, vacuous persistence is equivalent to persistence together with the absence of boundary steady states in the stoichiometric compatibility classes of $G$ that are not entirely contained in the boundary \cite[Proposition 5.2]{gnacadja-2011i}. In view of Proposition \ref{prop:persistent2bounded_persistent}\olr{v} and Proposition \ref{prop:cascades}, persistence of a cascade of PTM systems in our setting is equivalent to vacuous persistence.}
Futility implies that the connected components of the underlying substrate network of the cascaded PTM system are strongly connected \cite[Remark 4.6]{gnacadja-2011}. However, the condition is not necessary for futility. Therefore, our results establish that, for the overlapping class of cascades of PTM systems, strong connectedness of the components of the underlying substrate network is also necessary for vacuous persistence. We also note that our results are stated under more general kinetic assumptions.

\subsubsection{Proof of Proposition \ref{prop:cascades}}

Since $G$ is conservative, Proposition~\ref{prop:persistent2bounded_persistent} and Corollary \ref{cor:equivcrit} 
give the  implications
({\em iii}\,) $\Rightarrow$   ({\em iv}\,)  $\Rightarrow$ ({\em v}\,)   $\Rightarrow$  ({\em i}\,).

It remains to prove {({\em i}\,)} $\Rightarrow$ {({\em ii}\,)} and {({\em ii}\,)} $\Rightarrow$ {({\em iii}\,)}.
We begin with a few simple observations about signaling cascades of PTM systems. 

{Let $G$ be a signaling cascade of PTM systems with layers $G_1, \ldots, G_T$.
If the set of intermediates  $\Int$ is nonempty, let $G^*$ be the reaction network obtained by its removal. For each $i \in [T]$, let $G_i^*$ be the reaction network obtained from $G_i$ by the removal of the set of intermediates $\Int_i$.}

{By the next lemma,  we may assume without loss of generality that the cascade has no intermediates.}

\begin{lemma}\label{lemm:wlog}
	In the construction above, $G^*$ is a signaling cascade of PTM systems with layers $G_1^*, \ldots, G_T^*$. Furthermore, 
	\begin{itemize}
		\item[{{\em(}i\,{\em)}}] $G^*$ has the siphon/P-semiflow property if, and only if $G$ does also, and
		
		\item[{{\em(}ii\,{\em)}}] $G^*$ is consistent if, and only if $G$ is also.
	\end{itemize}
\end{lemma}

\begin{proof}
	For each $i \in [T]$, set $\Sub^*_i := Sub_i \cap \calS^*_i$, $\Enz^*_i := \Enz_i \cap \calS^*_i$, and $\Int^*_i := \varnothing$. Then $G_i^*$ is a PTM system with enzyme, substrate, and intermediates sets, respectively, $\Enz^*_i$, $\Sub^*_i$, and $\Int^*_i$, thus satisfying (F1). Properties (F2) and (F3) are inherited directly from $G$, and (F4) is trivial. This proves the first statement. Statements {({\em i}\,)} and {({\em ii}\,)} then follow directly from {Theorem \ref{thm:main}({\em ii}\,)--({\em iii}\,)}.
\end{proof}

Throughout the rest of this subsection, $G = (\calS, \calC, \calR)$ will be assumed to be a signaling cascade of PTM systems with an empty set of intermediates.

Next, let $G^- = (\calS^-, \calC^-, \calR^-)$ be the reaction network  with
\[
		\calS^- := \bigcup_{i = 1}^{T-1} \calS_i\,
		\qquad
		\calC^- := \bigcup_{i = 1}^{T-1} \calC_i\,
		\quad
		\text{and}
		\quad
		\calR^- := \bigcup_{i = 1}^{T-1} \calR_i\,.
\]
Set
	\[
		\Enz'_T := \Enz_T \cap \left( \bigcup_{i = 1}^{T - 1} \Enz_i \right)\,.
	\]
If $\Enz'_T \neq \varnothing$, then it is a set of catalysts of $G^-$. So, define $G'$ to be the network obtained from $G^-$ by the removal of the set of catalysts $\Enz'_T$.

\begin{lemma}\label{lemm:inductivehypothesis}
	In the construction above, $G'$ is a cascade of PTM systems with $T-1$ layers $G'_1, \ldots, G'_T$. Furthermore, for each $i \in [T-1]$, the underlying substrate networks of $G'_i$	and $G_i$ coincide.
\end{lemma}

\begin{proof}
	For each $i \in [T-1]$, define $\calR'_i$ to be the set of reactions $S + \alpha E \rightarrow S' + \alpha E \in \calR_i$ such that $S, S' \in \Sub_i$, $S \neq S'$, $E \in \Enz_i \backslash \Enz_T$, and $\alpha \in \{0, 1\}$, plus the reactions $S \rightarrow S'$ such that $S + E \rightarrow S' + E \in \calR_i$ for some $S, S' \in \Sub_i$, $S \neq S'$, and $E \in \Enz_T$. Then define $G'_i = (\calS'_i, \calC'_i, \calR'_i)$ to be the reaction network determined by $\calR'_i$. We then have
	\[
		\calS' = \bigcup_{i = 1}^{T-1} \calS'_i\,,
		\quad
		\calC' = \bigcup_{i = 1}^{T-1} \calC'_i\,,
		\quad
		\text{and}
		\quad
		\calR' = \bigcup_{i = 1}^{T-1} \calR'_i\,.
	\]
	Now $G'_i$ is a PTM system with $\Enz'_i = \Enz_i \backslash \Enz_T$, $\Sub'_i = \Sub_i$, and $\Int'_i = \varnothing$, $i = 1, \ldots, T-1$. Indeed, (M1) is fulfilled by construction, and (M2) and (M3) hold vacuously. Thus, (F1) holds. Furthermore, properties (F2) and (F3) are inherited from $G$, and (F4) is fulfilled vacuously. This shows $G'$ is a signaling cascade of PTM systems with layers $G'_1, \ldots, G'_{T-1}$.
	
	To establish the second statement, it is enough to show that $(\calR'_i)^{**} = \calR_i^{**}$, $i = 1, \ldots, T-1$. Let $i \in [T-1]$, and $S \rightarrow S' \in (\calR'_i)^{**}$. Then, by construction, $S + \alpha E \rightarrow S' + \alpha E \in \calR_i$ for some $S, S' \in \Sub_i$, $S \neq S'$, $E \in \Enz_i$, and $\alpha \in \{0, 1\}$, and so $S \rightarrow S' \in \calR_i^{**}$. Conversely, if $S \rightarrow S' \in \calR_i^{**}$, then $S + \alpha E \rightarrow S' + \alpha E \in \calR_i$ for some $S, S' \in \Sub_i$, $S \neq S'$, $E \in \Enz_i$, and $\alpha \in \{0, 1\}$. If $E \in \Enz_T$ and $\alpha = 1$, then we get $S \rightarrow S' \in \calR'_i$ by construction, and so $S \rightarrow S' \in (\calR'_i)^{**}$. Otherwise, $S + \alpha E \rightarrow S' + \alpha E \in \calR'_i$, and so $S \rightarrow S' \in (\calR'_i)^{**}$ after the removal of catalysts.
\end{proof}

Finally, let $\widehat G_T$ be the reaction network obtained from $G$ by the removal of the set of catalysts $\Enz_T$, and let $G^{**}_T$ be the underlying substrate network of $G_T$. 
Upon ordering the species and reactions of $\widehat G_T$ in such a way that all species belonging to $\Sub_T$ correspond to the bottom-most rows, and all monomolecular reactions between species in $\Sub_T$ correspond to the right-most columns, the stoichiometric matrix of $\widehat G_T$ may be written as
	
	\begin{equation}\label{eq:inductivehypothesis}
	\widehat N_T = \left[ \begin{array}{cc}
			N^\prime & 0 \\ 
			0 & N^{**}_T 
		\end{array}\right]\,,
		\end{equation}
where $N'$ is the stoichiometric matrix of the network $G'$ introduced above, and $N^{**}_T$ is the stoichiometric matrix of $G^{**}_T$. This decomposition will be used in the proofs of the next two results.

\medskip
\noindent {\textit{Proof of \ilr{i} $\Rightarrow$ \ilr{ii} in Proposition~\ref{prop:cascades}.}}
	We use induction on the number $T$ of layers. 
		For $T = 1$, this follows from Proposition \ref{prop:ptmpersistence}. 
	
	Now suppose the result holds for cascades of PTM systems with $T - 1$ layers for some $T \geqslant 2$, and let $G$ be a cascade with $T$ layers. 
	By Theorem \ref{thm:main2}({\em iii}\,), $\widehat G_T$ is consistent. So, there exists a $\widehat v_T \gg 0$ such that $\widehat N_T \widehat v_T = 0$. We may write $\widehat v_T = (v', v^{**}_T)$, where $v'$ corresponds to the reactions of $G'$, and $v^{**}_T$ corresponds to the reactions of $G^{**}_T$. From (\ref{eq:inductivehypothesis}), we obtain
	\[
		N'v' = 0
		\quad
		\text{and}
		\quad
		N^{**}_Tv^{**}_T = 0\,,
	\]
	concluding that $G'$ and $G^{**}_T$ are consistent. It follows by the inductive hypothesis, Lemma \ref{lemm:inductivehypothesis}, and Proposition \ref{prop:monomolecular_equivalence} that $G^{**}_1, \ldots, G^{**}_{T - 1}, G^{**}_T$, the underlying substrate networks of $G_1, \ldots, G_{T-1}, G_T$, respectively, are such that their connected components are strongly connected. This establishes the inductive step, proving the result.
	\mybox

\medskip
\noindent {\textit{Proof of \ilr{ii} $\Rightarrow$ \ilr{iii} in Proposition~\ref{prop:cascades}.}}
	We use induction on the number $T$ of layers.
		For $T = 1$, this follows from Proposition \ref{prop:ptmpersistence}.
	
	Now suppose the result holds for signaling cascades of PTM systems with $T - 1$ layers for some $T \geqslant 2$, and let $G$ be a cascade with $T$ layers. 
	By Theorem \ref{thm:main2}({\em ii}\,), it  is enough to show that $\widehat G_T$ has the siphon/P-semiflow property.
	
	By construction, the species set $\widehat\calS_T$ of $\widehat G_T$ can be partitioned as the disjoint union $\widehat\calS_T = \calS' \cup \calS^{**}_T$ of the species sets of $G'$ and $G^{**}_T$. We claim that every minimal siphon of $\widehat G_T$ is entirely contained in either $\calS'$ or $\calS^{**}_T$. To see this, let $\widehat \Sigma_T$ be any minimal siphon of $\widehat G_T$, and suppose it is not entirely contained in $\calS'$. So, $\Sigma_T \cap \calS^{**}_T \neq \varnothing$. By hypothesis, $G^{**}_T$ is a monomolecular network with the property that each of its connected components is strongly connected. Thus, each of its connected components is a minimal siphon. We conclude that $\Sigma_T$ contains one of the connected components of $G^{**}_T$ and, by minimality, must be actually equal to it.
	
	By the inductive hypothesis and Lemma \ref{lemm:inductivehypothesis}, $G'$ has the siphon/P-semiflow property. By Proposition \ref{prop:monomolecular_equivalence}, $G^{**}_T$ also has the siphon/P-semiflow property. We conclude from the block-diagonal decomposition in (\ref{eq:inductivehypothesis}) and the claim above that $\widehat G_T$ has the siphon/P-semiflow property.
	\mybox

\subsection{Dissipative Networks}\label{subsec:dissipative}

In the next definition, we use the same notation as in Subsection \ref{subsec:persistence}.

\begin{definition}[Dissipative Networks]\label{def:dissipative}
	A reaction network (\ref{eq:crntode}) is said to be {\em dissipative} if its solutions are eventually uniformly bounded. More precisely, if there exists a constant $K \geqslant 0$ such that
	\[
		\limsup_{t \to \infty} |\sigma(t, s_0)| \leqslant K\,,
	\]
	for each initial state $s_0 \geqslant 0$.
\mytriangle	
\end{definition}

\begin{corollary}\label{cor:dissipative2persistent}
If a dissipative reaction network is bounded-persistent, then it is persistent.	
\end{corollary}

\begin{proof}
Indeed, every solution of a dissipative reaction network is bounded. The conclusion then follows from Lemma \ref{lemm:bounded_persistent2persistent} {in Subsection \ref{app:prop2} in the appendix.}
\end{proof}

\begin{example}[Monomer-Dimer Toggle]\label{ex:dans}
Consider the monomer-dimer toggle model given by the reaction network
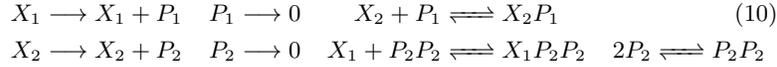
\begin{align}\label{eq:dans1}
		X_1 & \longrightarrow X_1 + P_1 & P_1 &  \longrightarrow 0  & X_2 + P_1& \longleftrightarrow X_2P_1   \\
		 X_2  & \longrightarrow X_2 + P_2 & P_2 &  \longrightarrow 0 & X_1 + P_2P_2 &  \longleftrightarrow X_1P_2P_2 & 2P_2 &\longleftrightarrow P_2P_2  \nonumber
\end{align}

The {leftmost} four reactions  model basal protein production and degradation. The $P_2P_2$ represents a dimeric species, while $X_2P_1$ and $X_1P_2P_2$ represent, respectively, monomers and dimers bound to gene promoters. See \cite[Page S1]{siegal-gaskins--franco--zhou--murray-2015} for further contextualization.

By removing the set of intermediates $\{X_2P_1, X_1P_2P_2\}$, 
we obtain the network
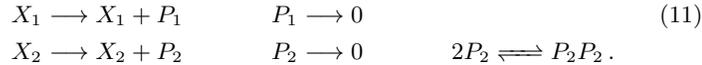
\begin{align}\label{eq:dans2}
X_1 & \longrightarrow X_1 + P_1 &  P_1 &  \longrightarrow 0   \\
X_2 &  \longrightarrow X_2 + P_2 & P_2 & \longrightarrow 0 &  2P_2 & \longleftrightarrow P_2P_2\,. \nonumber
\end{align}
Now $\{P_2P_2\}$ 
constitutes a set of intermediates of (\ref{eq:dans2}).
Its removal yields
\begin{equation}\label{eq:dans3}
		X_1 \longrightarrow X_1 + P_1 \quad\quad X_2 \longrightarrow X_2 + P_2 \quad\quad P_1 \longrightarrow 0 \quad\quad P_2 \longrightarrow 0
\end{equation}
Now $ \{X_1, X_2\}$
is a set of catalysts of (\ref{eq:dans3}). Their removal leaves us with
\begin{equation}\label{eq:dans4}
	P_1 \longleftrightarrow 0 \longleftrightarrow P_2\,.
\end{equation}
This is a non-conservative  strongly connected monomolecular network. {Thus, by Corollary~\ref{cor:intcat}, the network  (\ref{eq:dans1}) is bounded-persistent and does not have boundary steady states in any stoichiometric compatibility class that is not already contained in the boundary of the positive orthant.} 
Under mass-action kinetics, (\ref{eq:dans1}) is dissipative \cite[Pages S7--S8]{siegal-gaskins--franco--zhou--murray-2015}. Thus, it is also persistent by Corollary \ref{cor:dissipative2persistent}. 
\mydiamond	
\end{example}

\subsection{A Shuttling and Degradation Focused Wnt Model}\label{subsec:bss}

The following reaction network model for the Wnt pathway was proposed in \cite{maclean--rosen--byrne--harrington-2015}.
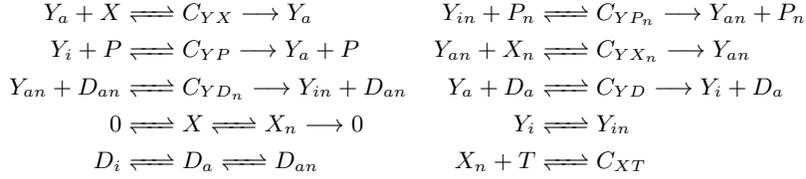
		\begin{align*}
			Y_a + X & \longleftrightarrow C_{YX} \longrightarrow Y_a & Y_{in} + P_n & \longleftrightarrow C_{YP_n} \longrightarrow Y_{an} + P_n  \\
			Y_i + P & \longleftrightarrow C_{YP} \longrightarrow Y_a + P  &Y_{an} + X_n & \longleftrightarrow C_{YX_n} \longrightarrow Y_{an} \\
			Y_{an} + D_{an} & \longleftrightarrow C_{YD_n} \longrightarrow Y_{in} + D_{an} & Y_a + D_a & \longleftrightarrow C_{YD} \longrightarrow Y_i + D_a \\
			0 & \longleftrightarrow X \longleftrightarrow X_n \longrightarrow 0  & Y_i & \longleftrightarrow Y_{in}  \\
			D_i & \longleftrightarrow D_a \longleftrightarrow D_{an} & X_n + T & \longleftrightarrow C_{XT}  
		\end{align*}
Note that $\{C_{YX}, C_{YP_n}, C_{YP}, C_{YX_n}, C_{YD_n}, C_{YD}, C_{XT}, D_i\}$ is a set of intermediates. Their removal yields the reaction network

	\noindent
	\begin{minipage}[c]{0.45\linewidth}
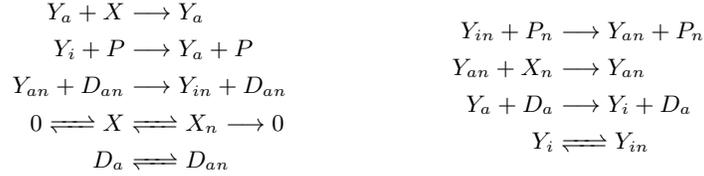
	\[
		\begin{aligned}
			Y_a + X & \longrightarrow Y_a \\
			Y_i + P & \longrightarrow Y_a + P \\
			Y_{an} + D_{an} & \longrightarrow Y_{in} + D_{an} \\
			0 \longleftrightarrow X & \longleftrightarrow X_n \longrightarrow 0 \\
			D_a & \longleftrightarrow D_{an}
		\end{aligned}
	\]
	\end{minipage}%
	\hfill
	\begin{minipage}[c]{0.45\linewidth}
	\[
		\begin{aligned}
			Y_{in} + P_n & \longrightarrow Y_{an} + P_n \\
			Y_{an} + X_n & \longrightarrow Y_{an} \\
			Y_a + D_a & \longrightarrow Y_i + D_a \\
			Y_i & \longleftrightarrow Y_{in}
		\end{aligned}
	\]
	\end{minipage}
	
	\vskip 1ex

\noindent Now $\{D_a, D_{an}, P_n, P\}$ constitutes a set of catalysts. After their removal, we obtain the reaction network
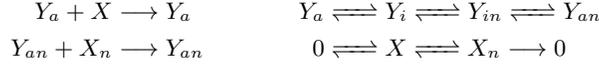
\begin{align*}
Y_a + X &  \longrightarrow Y_a & Y_a &  \longleftrightarrow Y_i \longleftrightarrow Y_{in} \longleftrightarrow Y_{an} \\
Y_{an} + X_n & \longrightarrow Y_{an} & 	0 & \longleftrightarrow X \longleftrightarrow X_n \longrightarrow 0
\end{align*}
We may now remove $\{Y_i, Y_{in}\}$ as a set of intermediates, and then remove $\{Y_a, Y_{an}\}$ as a set of catalysts, thus obtaining
\[
	0 \longleftrightarrow X \longleftrightarrow X_n \longrightarrow 0\,.
\]
This is a non-conservative strongly connected monomolecular network. 
{Thus, by Corollary~\ref{cor:intcat}, the network  (\ref{eq:dans1}) is bounded-persistent and does not have boundary steady states in any stoichiometric compatibility class that is not already contained in the boundary of the positive orthant.} 
\mydiamond

\section{Proofs of Theorems \ref{thm:main}, \ref{thm:main2} and \ref{thm:primitive}}\label{sec:proofofmainresult}

In Subsection \ref{subsec:intermediates_proof} we prove Theorem \ref{thm:main}, items ({\em i}\,), ({\em iii}\,) and ({\em iv}\,). 
Likewise, the proof of Theorem \ref{thm:main2}, items   ({\em i}\,), ({\em iii}\,) and ({\em iv}\,) is carried out in Subsection \ref{subsec:catalysts_proof}.
Item ({\em ii}\,) in each result 
follows from ({\em i}\,) in the same result by Corollary~\ref{cor:equivcrit}.
{The proof of Theorem \ref{thm:primitive} is worked out in Subsection \ref{subsec:unique_primitive_reduction}.}

{%
We begin with a general fact about reaction networks. Let $G = (\calS, \calC, \calR)$ be a reaction network, and let $(\calC_1, \calR_1), \ldots, (\calC_J, \calR_J)$ be the connected components of its reaction graph $(\calC, \calR)$.
\begin{lemma}\label{lemm:components1}
	For each $j \in [J]$, $y' - y \in \Gamma\,$ for all $y, y' \in \calC_j\,$.
\end{lemma}

\begin{proof}
	Since $y$ and $y'$ are in the same connected component of $(\calC, \calR)$, there exists an undirected reaction path
$y \ \textrm{---} \ y_1 \ \textrm{---} \ \cdots \ \textrm{---} \ y_k \ \textrm{---} \ y'$
	in $(\calC, \calR)$ connecting $y$ and $y'$. 
	Now
	\[
		y' - y = (y' - y_1) + \sum_{i = 2}^{k} (y_{i - 1} - y_i) + (y_k - y)\in \Gamma \,,
	\]
	establishing the lemma.
\end{proof}}

\subsection{Intermediates}\label{subsec:intermediates_proof}

Now suppose $G^* = (\calS^*, \calC^*, \calR^*)$ is the reduction of $G$ by the removal of a set of intermediates $\calY$. Recall that $\calS^*$ does not always agree with $\calS \backslash \calY$. Let
\[
	\calX := (\calS \backslash \calY) \backslash \calS^*\,,
\]
and write
\[
	\calX = \{X_1, \ldots, X_\ell\}\,,
\qquad\text{and}\qquad
	\calS^* = \{S^*_1, \ldots, S^*_n\}\,.
\]
Thus,
\[
	\calS = \calS^* \cup \calX \cup \calY = \{S^*_1, \ldots, S^*_n, X_1, \ldots, X_\ell, Y_1, \ldots, Y_p\}\,.
\]
This is the ordering we shall assume whenever working with the stoichiometric matrix or the stoichiometric subspace of $G$. Given a complex
\[
	y 
	= (\alpha_1, \ldots, \alpha_n, \gamma_1, \ldots, \gamma_\ell, \beta_1, \ldots, \beta_p)
	= \sum_{i = 1}^{n} \alpha_iS^*_i + \sum_{i = 1}^{\ell} \gamma_iX_i + \sum_{i = 1}^{p} \beta_iY_i\,,
\]
in $\calC$, we will denote its projection over the first $n$ coordinates by
\begin{equation}\label{eq:pry}
	\pry := (\alpha_1, \ldots, \alpha_n) = \sum_{i = 1}^{n} \alpha_iS^*_i\,.
\end{equation}
Conversely, given a complex $\pry$ in $\mathbb{R}^n_{\geq 0}$ as in \eqref{eq:pry}, 
we denote its embedding in $\r^{n + \ell + p}$ by
\[
	y := (\alpha_1, \ldots, \alpha_n, 0, \ldots, 0) = \sum_{i = 1}^{n} \alpha_iS^*_i\,.
\]

\begin{lemma}\label{lemm:components2}
	For each $j \in [J]$, if $y, y' \in \calC_j \backslash \calY$, then $\pry$ and $\pry'$ are in the same connected component of $(\calC^*, \calR^*)$.
\end{lemma}

\begin{proof} If  $y \neq y'$, then there exists an undirected {reaction} path
	\[
		y \ \textrm{---} \ y_1 \ \textrm{---} \ \cdots \ \textrm{---} \ y_k \ \textrm{---} \ y'
	\]
	in $(\calC, \calR)$ connecting $y$ and $y'$.
	Let $i_1, \ldots, i_d \in [k]$ be the indices such that $y_{i_1}, \ldots, y_{i_d} \in \calC \backslash \calY$, so that each non-endpoint in each of the paths
	\begin{equation}\label{eq:last}
		\begin{aligned}
			y \ \textrm{---} \ y_1 \ \textrm{---} \ &\cdots \ \textrm{---} \ y_{i_1 - 1} \ \textrm{---} \ y_{i_1}\,, \\
			y_{i_1} \ \textrm{---} \ y_{i_1 + 1} \ \textrm{---} \ &\cdots \ \textrm{---} \ y_{i_2 - 1} \ \textrm{---} \ y_{i_2}\,, \\
			&\hspace{0.625em} \vdots \\
			y_{i_d} \ \textrm{---} \ y_{i_d + 1} \ \textrm{---} \ &\cdots  \ \textrm{---} \ y_k \ \textrm{---} \ y'\,, 
		\end{aligned}
	\end{equation}
	is an intermediate. 
		We may assume without loss of generality that, within each path, all arrows point in the same direction. To see this, suppose that this is not the case for, say, the first path. Suppose $y \rightarrow y_1$, and let ${q_1} \in [i_1 - 1]$ be the index corresponding to the first (intermediate) complex where the arrows switch directions. So we have
	$$	y \longrightarrow y_{1} \longrightarrow \cdots \longrightarrow y_{q_1} \longleftarrow y_{q_1 + 1} \cdots\,.
	$$
	
	By (I2), there exists $y^{(1)} \in \calC \backslash \calY$, and $Y^{(1)}, \ldots, Y^{(p_1)} \in \calY$ such that,
	\[
		y_{q_1} \longrightarrow Y^{(1)} \longrightarrow \cdots \longrightarrow Y^{(p_1)} \longrightarrow y^{(1)}
	\]
	is a reaction path in $G$. 
	We may then split the first path in (\ref{eq:last}) into the two paths
\begin{align*}
		y  \longrightarrow y_1 \longrightarrow \cdots \longrightarrow y_{q_1} \longrightarrow Y^{(1)} \longrightarrow \cdots \longrightarrow Y^{(p_1)} \longrightarrow y^{(1)}\,, \\
		y^{(1)}  \longleftarrow Y^{(p_1)} \longleftarrow \cdots \longleftarrow Y^{(1)} \longleftarrow y_{q_1} \longleftarrow y_{q_1 + 1} \cdots\,,
\end{align*}
	{where in each path we remove any loops starting and ending at an intermediate that might have been created.}
	If there are other changes of direction between $y_{q_1}$ and $y_{i_1}$, we may employ the same construction as many times as needed.
	If $y \leftarrow y_1$ instead, the argument is analogous, and the same construction applies also to any other path not having the property that all arrows point in the same direction.
	
This construction gives an undirected {reaction} path in $(\calC^*, \calR^*)$:
	\[
		\pry \ \textrm{---} \ \pry_{i_1} \ \textrm{---} \ \cdots \ \textrm{---} \ \pry_{i_d} \ \textrm{---} \ \pry'.
	\]
We conclude that $\pry$ and $\pry'$ are in the same connected component of $(\calC^*, \calR^*)$.
\end{proof}

\begin{remark}\label{rem:clarification}
{We have that $\supp y \cap \calX  = \varnothing$ whenever the connected component of $G$ that contains $y$ contains at least one more non-intermediate complex $y'$.}
\mybox
\end{remark}

\subsubsection*{Conservation Laws}

We prove here Theorem~\ref{thm:main}\olr{iv}.  
In what follows, $\Gamma^* \subseteq \r^n$ is the stoichiometric subspace of $G^*$. Thus, its orthogonal complement $(\Gamma^*)^\perp$ is taken in $\r^n$.

\begin{lemma}\label{lemm:a}
	For each $j \in [J]$,
	\[
		(\omega^*, x, 0) \cdot y = (\omega^*, x, 0) \cdot y'\,,
		\quad
		\forall y, y' \in \calC_j \backslash \calY\,,
		\quad
		\forall
		(\omega^*, x) \in (\Gamma^*)^\perp \times \r^\ell\,.
	\]	
\end{lemma}

\begin{proof}
	Fix arbitrarily $\omega^* \in (\Gamma^*)^\perp$, $x \in \r^\ell$, $j \in [J]$, and ${y,y'} \in \calC_j \backslash \calY$. The equality is trivial if $y = y'$, so, assume $y \neq y'$.
	By Lemma \ref{lemm:components2}, $\pry$ and $\pry'$ are in the same connected component of $(\calC^*, \calR^*)$. By Lemma \ref{lemm:components1}, we conclude that $\pry - \pry' \in \Gamma^*$. 
	 {In view of Remark \ref{rem:clarification}, we have now
		\[
			(\omega^*, x, 0) \cdot (y - y') = \omega^* \cdot (\pry - \pry') =0\,,
		\]}
	completing the proof.
\end{proof}

For each $j \in [J]$, fix arbitrarily a complex $y_j \in \calC_j \backslash \calY$. Property (I2) in the definition of intermediates ensures that $\calC_j \backslash \calY$ is always nonempty. For each $i \in [p]$, let $j_i \in [J]$ be the index uniquely defined by the property that $Y_i \in \calC_{j_i}$. Define
\[
	\begin{array}{rcl}
		a\colon (\Gamma^*)^\perp \times \r^\ell & \longrightarrow & \r^{p} \\[1ex]
		(\omega^*, x) & \longmapsto & ((\omega^*, x, 0) \cdot y_{j_1}, \ldots, (\omega^*, x,0 ) \cdot y_{j_p}).
	\end{array}
\]
Note that, by Lemma \ref{lemm:a}, $a$ is independent of the chosen representatives $y_j \in \calC_j \backslash \calY$, $j \in [J]$.

\begin{lemma}\label{lemm:intermediatecls1og2}
	Suppose $G^* = (\calS^*, \calC^*, \calR^*)$ is the reduction of $G = (\calS, \calC, \calR)$ by the removal of a set of intermediates $\calY$. Then
	\[
		\Gamma^\perp = {\{(\omega^*, x, a(\omega^*,x))\,|\ (\omega^*, x) \in (\Gamma^*)^\perp \times \r^\ell\}}\,.
	\] 
\end{lemma}

\begin{proof}
	(I) We first show the inclusion $\supseteq$.
	To this end, fix arbitrarily $\omega^* \in (\Gamma^*)^\perp$, and $x \in \r^\ell$. Denote
	\[
		\omega := (\omega^*, x, a(\omega^*,x))\,.
	\]
	Fix arbitrarily $y \rightarrow y' \in \calR$. We want to show that
	\[
		\omega \cdot (y' - y) = 0\,.
	\]
	There are four possibilities.
	
	(1) If $\pry \rightarrow \pry' \in \calR^*$, then
	\[
		\omega \cdot (y' - y) = \omega^* \cdot (\pry' - \pry) = 0\,.
	\]
	
	(2) If $y \rightarrow y' = y \rightarrow Y_i$ for some $Y_i \in \calY$, and some $y \in \calC \backslash \calY$, then
	\[
		\omega \cdot (y' - y)
		= \omega \cdot (Y_i - y) 
		= (\omega^*, x, 0) \cdot y_{j_i} - (\omega^*, x, 0) \cdot y 
		= 0\,,
	\]
	{by Lemma \ref{lemm:a}}, since $y$ and $y_{j_i}$ belong to the same connected component of $(\calC, \calR)$.
	
	(3) If $y \rightarrow y' = Y_i \rightarrow y'$ for some $Y_i \in \calY$, and $y \in \calC \backslash \calY$, then the argument is the same as in (2).
	
	(4) If $y \rightarrow y' = Y_i \rightarrow Y_k$ for some $Y_i, Y_k \in \calY$, then $j_i = j_k$, and so
	\[
		\omega \cdot (y' - y)
		= \omega \cdot (Y_k - Y_i) = (\omega^*, x, 0) \cdot y_{j_i} - (\omega^*, x, 0) \cdot y_{j_k} = 0\,.
	\]
	
	This establishes  the inclusion $\supseteq$. In particular,
	\[
		\dim ((\Gamma^*)^\perp \times \r^\ell) = \dim (\Gamma^*)^\perp + \ell \leqslant \dim \Gamma^\perp\,.
	\]
	
	(II) To finish the proof, it is now enough to show that
	\begin{equation}\label{eq:converse1}
		\dim \Gamma^\perp \leqslant \dim (\Gamma^*)^\perp + \ell\,.
	\end{equation}
	We claim that
	\begin{equation}\label{eq:converse2}
		\dim \Gamma \geqslant \dim \Gamma^* + p\,.
	\end{equation}
	On the one hand, for each reaction $\pry \rightarrow \pry' \in \calR^*$, there exists a reaction path in $G$ connecting $y$ to $y'$, so $y$ and $y'$ are in the same connected component of $(\calC, \calR)$. It follows by Lemma \ref{lemm:components1} that $y' - y \in \Gamma$. On the other hand, for each intermediate $Y_i \in \calY$, there exists a $y^{(i)} \in \calC \backslash \calY$ and a reaction path in $G$ connecting $Y_i$ to $y^{(i)}$. Again by Lemma \ref{lemm:components1}, we have $Y_i - y^{(i)} \in \Gamma$. Furthermore, $Y_1 - y^{(1)}, \ldots, Y_p - y^{(p)}, (y' - y)$ are linearly independent for each $\pry \rightarrow \pry' \in \calR^*$. This gives us (\ref{eq:converse2}).
	
	Combining $\dim \Gamma + \dim \Gamma^\perp = \dim \r^{n + \ell + p}$ and (\ref{eq:converse2}), we get
	\[
		\dim \Gamma^\perp = n + \ell + p - \dim \Gamma \leqslant n - \dim \Gamma^* + \ell = \dim (\Gamma^*)^\perp + \ell\,.
	\]
	This establishes (\ref{eq:converse1}), completing the proof.
\end{proof}

\noindent {\textit{Proof of Theorem  \ref{thm:main}\ilr{iv}.}}
	{($\Rightarrow$) If $\omega = (\omega^*, x, a(\omega^*, x))$ is a strictly positive conservation law of $G$, then $\omega^*$ is a strictly positive conservation law of $G^*$.}
	
	{($\Leftarrow)$ If $\omega^*$ is a strictly positive conservation law of $G^*$, then choose any $x\in \mathbb{R}^\ell_{>0}$.  It holds that $(\omega^*, x, 0) \cdot y_{j_i}>0$ for all $i\in [p]$ since {$y_{j_i} \neq 0$ and} the support of $y_{j_i}$ is included in the support of $(\omega^*,x)$. 	Then $\omega = (\omega^*, x, a(\omega^*, x))$ is a strictly positive conservation law of $G$ by Lemma~\ref{lemm:intermediatecls1og2}.}
	\mybox

\subsubsection*{Siphons}

\begin{lemma}\label{lemm:minsiphons1}
	Suppose $G^* = (\calS^*, \calC^*, \calR^*)$ is the reduction of $G = (\calS, \calC, \calR)$ by the removal of a set of intermediates $\calY$. If $\Sigma$ is a siphon of $G$, then
	\[
		\Sigma^* := \Sigma \cap \calS^*
	\]
	is either the empty set, or a siphon of $G^*$. Furthermore, if $\Sigma^*$ is empty, then $\Sigma \cap \calX$ is nonempty.

\end{lemma}

\begin{proof}
	First suppose $\Sigma^* \neq \varnothing$. Pick any $S' \in \Sigma^*$, and let $\pry \rightarrow \pry' \in \calR^*$ be any reaction having $S'$ as one of its products. Then there exists a reaction path in $G$ connecting $y$ and $y'$. Since $\Sigma$ is a siphon of $G$,  some species $S$ constituting $y$ belongs to $\Sigma$. Since $\pry \in \calC^*$, we must have $S \in \Sigma^*$. Thus, $\Sigma^*$ is a siphon of $G^*$.
	
	Now suppose $\Sigma^* = \varnothing$. Since $\Sigma \neq \varnothing$ and $\calS = \calS^* \cup \calX \cup \calY$, we must have $\Sigma \cap \calX \neq \varnothing$ or $\Sigma \cap \calY \neq \varnothing$. If $\Sigma \cap \calX \neq \varnothing$, then we have nothing left to prove. So, assume $\Sigma \cap \calY \neq \varnothing$, and fix arbitrarily a $Y \in \Sigma \cap \calY$. By (I2), there exist $y \in \calC \backslash \calY$ and a reaction path in $G$ connecting $y$ to $Y$, and we conclude that one of the species in $y$ belongs to $\Sigma$. Since $y$ is supported in $\calS^* \cup \calX$, and since $\Sigma \cap \calS^* = \varnothing$ by hypothesis, we conclude that $\Sigma \cap \calX \neq \varnothing$.
\end{proof}

Given a subset $\Sigma \subseteq \calS$, we define $M(\Sigma)$ to be the subset of intermediates $Y \in \calY$ that appear in a reaction path
	\[
		Y \longrightarrow Y^{(1)} \longrightarrow \cdots \longrightarrow Y^{(k)} \longrightarrow y'
	\]
	for some $y'$ whose support intersects $\Sigma$, and some $Y^{(1)}, \ldots, Y^{(k)} \in \calY$.
{Note that if $\Sigma$ is a siphon of $G$, then $M(\Sigma)\subseteq \Sigma$ by Remark~\ref{rk:siphonpath}.}

\begin{lemma}\label{lemm:minsiphons2}
	Suppose $G^* = (\calS^*, \calC^*, \calR^*)$ is the reduction of $G = (\calS, \calC, \calR)$ by the removal of a set of intermediates $\calY$. If $\Sigma^*$ is a siphon of $G^*$, then
	\[
		\Sigma := \Sigma^* \cup \intclr(\Sigma^*)
	\]
	is a siphon of $G$. Furthermore, any siphon of $G$ containing $\Sigma^*$ must also contain $M(\Sigma^*)$.
\end{lemma}

\begin{proof}
	Pick any $S' \in \Sigma$, and let $y \rightarrow y' \in \calR$ be any reaction having $S'$ as one of its products.
	
	Suppose first $S' \in \Sigma^*$. If $\pry \rightarrow \pry' \in \calR^*$, then $\Sigma^*$ contains some reactant of $y \rightarrow y'$, and so does $\Sigma$. If $\pry \rightarrow \pry' \notin \calR^*$, then $y \rightarrow y' = Y \rightarrow y'$ for some $Y \in \calY$. By {construction}, $Y \in M(\Sigma^*) \subseteq \Sigma$.
	
	Now suppose $S' \notin \Sigma^*$. Then $S' \in M(\Sigma^*)$, meaning that $y' = S' \in \calY$, and that there exists a reaction path
	\[
		S' \longrightarrow Y^{(1)} \longrightarrow \cdots \longrightarrow Y^{(k)} \longrightarrow y_0'
	\]
	in $G$ such that {the support of $y_0'$ intersects $\Sigma^*$}, and $Y^{(1)}, \ldots, Y^{(k)} \in \calY$. If $y \in \calY$, then it follows that $y \in M(\Sigma^*)$, and so $y \rightarrow y'$ has a reactant in $\Sigma$. If $y \notin \calY$, we have $\pry \rightarrow \pry_0' \in \calR^*$, and so one of the species constituting $\pry$ belongs to $\Sigma^*$. We conclude that one of the reactants of $y \rightarrow y'$ belongs to $\Sigma$. This completes the proof that $\Sigma$ is a siphon of $G$.
	
	It follows straight from the construction of $M(\Sigma^*)$ and {Remark~\ref{rk:siphonpath}}, that any siphon of $G$ containing $\Sigma^*$ must also contain all the intermediates in $M(\Sigma^*)$.
\end{proof}

\subsubsection*{Drainable and Self-Replicable Siphons}

\begin{lemma}\label{lemm:drainable_Y_aux}
Suppose $G^* = (\calS^*, \calC^*, \calR^*)$ is the reduction of $G = (\calS, \calC, \calR)$ by the removal of a set of intermediates $\calY$. If $\Sigma \subseteq \calS$ is a drainable or self-replicable siphon of $G$, then $\Sigma^* := \Sigma \cap \calS^*$ is nonempty.	
\end{lemma}

\begin{proof}
Suppose on contrary that $\Sigma^* = \varnothing$. Then $\Sigma \cap \calX$
	is nonempty {by Lemma~\ref{lemm:minsiphons1}}. Consider the P-semiflow
	\[
		\omega := (0, x, a(0, x)) > 0\,,
\qquad \text{where}\qquad 
		x := \sum_{i\colon X_i \in \Sigma \cap \calX} {\mathbf e}_i\,.
	\]
{Let 	$i \in [p]$ such that $a_i(0, x) = (0, x, 0) \cdot y_{j_i} > 0$.
Then $\supp(y_{j_i})\cap (\Sigma \cap \calX) \neq \varnothing$. By Remark~\ref{rem:clarification}, $y_{j_i}$ is the only non-intermediate complex in the connected component where $Y_i$ belongs to. Therefore there must exist a reaction path connecting $Y_i$ to $y_{j_i}$ and hence $Y_i\in M(\Sigma \cap \calX)$.}
  It follows that $\supp \omega \subseteq (\Sigma \cap \calX) \cup M(\Sigma \cap \calX) \subseteq \Sigma$. So, $\Sigma$ contains the support of a P-semiflow, in other words, $\Sigma$ is noncritical. We conclude by Proposition~\ref{prop:connection}\olr{i}
 that $\Sigma$ can be neither drainable nor self-replicable, contradicting the hypotheses.	
\end{proof}

\newcommand{\prz}{\widehat{z}}

We are now ready to prove Theorem  \ref{thm:main}\olr{i}.

\medskip
\noindent {\textit{Proof of Theorem  \ref{thm:main}\ilr{i}.}}
	In virtue of Lemma \ref{lemm:intermediateinduction}, it suffices to show that the result holds for the removal of a single intermediate $Y$. The general result then follows by induction on the size of the set of intermediates.
	
	$(\Leftarrow)$ Suppose $\Sigma^* \subseteq \calS^*$ is a drainable siphon of {$G^*$}. By Lemma \ref{lemm:minsiphons2}, $\Sigma := \Sigma^* \cup M(\Sigma^*)$ is a siphon of $G$. We will show that it is drainable.
	
	By construction, we have $M(\Sigma^*) = \varnothing$ or $M(\Sigma^*) = \{Y\}$. If $M(\Sigma^*) = \varnothing$, then $\Sigma = \Sigma^*$,  {and any reaction that contains a species in $\Sigma^*$ in the product belongs to $\calR^*_c$. }Thus we have nothing left to show. So, we may assume $M(\Sigma^*) = \{Y\}$. Since $\Sigma^*$ is drainable, there exist reactions
		\[
			\pry_1 \rightarrow \pry_1', \ldots, \pry_{k_Y} \rightarrow \pry_{k_Y}' \in {\calR^*_Y}
\qquad \text{and}\qquad
			\prz_1 \rightarrow \prz_1', \ldots, \prz_{k_c} \rightarrow \prz_{k_c}' \in {\calR^*_c}
		\]
		such that
		\[
			\theta_i :=  \sum_{j = 1}^{k_Y} (\pry_j' - \pry_j)_i + \sum_{j = 1}^{k_c} (\prz_j' - \prz_j)_i < 0\,,
		\quad
		\forall i \in [n]\,\colon\ S_i^* \in \Sigma^*\,.
		\]
		Let $T$ be a large enough positive integer such that { $(\pry_1')_i + T\theta_i <0$} 
for every $i \in [n]$ such that $S_i^* \in {\Sigma^*}$. We have
		\[
			y_1 \rightarrow Y,\ Y \rightarrow y_1', \ldots, y_{k_Y} \rightarrow Y,\ Y \rightarrow y_{k_Y}', \ z_1 \rightarrow z_1', \ldots, z_{k_c} \rightarrow z_{k_c}' \in \calR\,.
		\]
Let
{%
$$\alpha :=   (y_1' - Y) + \sum_{j = 1}^{k_Y} T(Y - y_j) + \sum_{j = 1}^{k_Y} T(y_j' - Y) + \sum_{j = 1}^{k_c} T(z_j' - z_j) $$
}
{%
We have
$$ \alpha_i = \begin{cases}
(\pry_1')_i + T\theta_i <0  & \text{if }i\in [n] \colon S_i^* \in \Sigma^*\\[1ex]
-1<0  & \text{if } i = n+\ell+1.
\end{cases}$$
The last case corresponds to the coordinate of $Y$.	}This shows $\Sigma$ is drainable.

	$(\Rightarrow)$ Suppose $\Sigma \subseteq \calS$ is a drainable siphon of $G$. By Lemmas \ref{lemm:minsiphons1} and \ref{lemm:drainable_Y_aux}, $\Sigma^* := \Sigma \cap \calS^*$ is a siphon of $G^*$. We will show that it is drainable.
	
	Since $\Sigma$ is drainable, there exist reactions 
	\begin{equation}\label{eq:bora}
	y_1 \rightarrow Y, \ldots, y_{k_Y} \rightarrow Y, Y \rightarrow y_1', \ldots, Y \rightarrow y_{k_d}', z_1 \rightarrow z'_1, \ldots, z_{k_c} \rightarrow z'_{k_c} \in \calR
	\end{equation}
	such that $y_1, \ldots, y_{k_Y}, y_1', \ldots, y_{k_d}', z_1, z_1', \ldots, z_{k_c}, z_{k_c}' \in \calC \backslash \calY$, and
{$$
\alpha_i:= \sum_{j = 1}^{k_Y} (Y - y_j)_i + \sum_{j = 1}^{k_d} (y_j' - Y)_i + \sum_{j = 1}^{k_c} (z_j' - z_j)_i  < 0\,,
$$
for all $i\in [n]$ such that $S_i^* \in \Sigma^*=\Sigma\cap \calS^*$, and also for $i=n+\ell+1$, if $Y\in \Sigma$.
In the latter case, since $\alpha_{n+\ell+1}= k_Y-k_d$, it must hold that $k_Y<k_d$.

If $k_d > k_Y$, then 
$$ \alpha_i =  \sum_{j = 1}^{k_Y} (y_j' - y_j)_i + \sum_{j = k_c+1}^{k_d} (y_j')_i + \sum_{j = 1}^{k_c} (z_j' - z_j)_i 
\geq  \sum_{j = 1}^{k_Y} (\pry_j' - \pry_j)_i + \sum_{j = 1}^{k_c} (\prz_j' - \prz_j)_i\,, $$
for all $i$ such that $S_i^* \in \Sigma^*$. Thus $\Sigma^*$ is drainable using the reactions
$\pry_1 \rightarrow \pry_1', \ldots, \pry_{k_Y} \rightarrow \pry_{k_Y}' \in {\calR^*_Y}$ and $
			\prz_1 \rightarrow \prz_1', \ldots, \prz_{k_c} \rightarrow \prz_{k_c}' \in {\calR^*_c}$.

If $k_Y \geq k_d$, then necessarily $Y\notin\Sigma$. Then, for all reactions in (\ref{eq:bora}) of the form $Y\rightarrow y'$, the support of $y'$ does not contain any species in $\Sigma^*$. In particular this holds for $y_1',\dots,y_{k_d}'$. Choose any $y' \in \calC \backslash \calY$ such that $Y \rightarrow y' \in \calR$.
We have
$$ \alpha_i = \sum_{j = 1}^{k_Y} - (y_j)_i + \sum_{j = 1}^{k_Y} (y')_i + \sum_{j = 1}^{k_c} (z_j' - z_j)_i  =  \sum_{j = 1}^{k_Y} (\pry' - \pry_j)_i + \sum_{j = 1}^{k_c} (\prz_j' - \prz_j)_i\,,  $$
for all $i$ such that $S_i^* \in \Sigma^*$, since $(y')_i=0$. 
Thus $\Sigma^*$ is drainable using the reactions
$\pry_1 \rightarrow \pry', \ldots, \pry_{k_Y} \rightarrow \pry' \in {\calR^*_Y}$ and $
			\prz_1 \rightarrow \prz_1', \ldots, \prz_{k_c} \rightarrow \prz_{k_c}' \in {\calR^*_c}$.
	}
		
The proof for self-replicable siphons is entirely analogous, with the appropriate inequalities reversed, and the roles played by reactions creating or consuming $Y$ swapped.
		\mybox

\subsubsection*{Consistency}

\begin{lemma}\label{lemm:consistencypreserved}
Suppose $G^* = (\calS^*, \calC^*, \calR^*)$ is the reduction of a reaction network $G = (\calS, \calC, \calR)$ by the removal of a set of intermediates $\{Y\}$ containing a single intermediate $Y$. Then $G^*$ is consistent if, and only if network $G$ is consistent.
\end{lemma}

\begin{proof}
($\Rightarrow$) Suppose $G^*$ is consistent. This is equivalent to say that
	\[
		\sum_{y \rightarrow y' \in \calR^*} v_{y \rightarrow y'}(y' - y) = 0
	\]
for some $v_{y \rightarrow y'} > 0$, $y \rightarrow y' \in \calR^*$. Let ${\calR}_Y^* \subseteq \calR^*$ be the subset of reactions $y \rightarrow y' \in \calR^*$ such that $y \rightarrow Y, Y \rightarrow y' \in \calR$, and let $\calR^*_c \subseteq \calR^*$ be the subset of all reactions $y \rightarrow y' \in \calR^*$ such that $y \rightarrow y' \in \calR$. Note that
	\[
		\calR^* = \calR^*_Y \cup \calR^*_c\,,
	\]
and that the union need not be disjoint. Let $\calC_{\leftrightarrow} \subseteq \calC$ be the subset of complexes $y \in \calC$ such that $y \rightarrow Y, Y \rightarrow y \in \calR$. Then
	\[
	\arraycolsep=0.25em
	\begin{array}{rcl}
		0
		&=& \displaystyle \left( \sum_{y \rightarrow y' \in \calR^*_Y \backslash \calR^*_c} 
			+ \sum_{y \rightarrow y' \in \calR^*_Y \cap \calR^*_c} 
			+ \sum_{y \rightarrow y' \in \calR^*_c \backslash \calR^*_Y} \right) v_{y \rightarrow y'}(y' - y) \\[4ex]
		& & + \displaystyle \sum_{y \in \calC_{Y}} (Y - y)
			+ \sum_{y' \in \calC_{Y}} (y' - Y) \\[4ex]
		&=& \left( \displaystyle \frac{1}{2}\sum_{y \rightarrow y' \in \calR^*_Y \cap\calR^*_c} 
			+ \sum_{y \rightarrow y' \in \calR^*_c \backslash \calR^*_Y} \right) v_{y \rightarrow y'}(y' - y) \\[4ex]
		& & + \left( \displaystyle \sum_{y \rightarrow y' \in \calR^*_Y \backslash \calR^*_c} v_{y \rightarrow y'} 
			+ \frac{1}{2}\sum_{y \rightarrow y' \in \calR^*_Y \cap \calR^*_c} v_{y \rightarrow y'} 
			+ \sum_{y \in \calC_Y} \right) (Y - y) \\[4ex]
		& & + \left( \displaystyle \sum_{y \rightarrow y' \in \calR^*_Y \backslash \calR^*_c} v_{y \rightarrow y'} 
			+ \frac{1}{2}\sum_{y \rightarrow y' \in \calR^*_Y \cap \calR^*_c} v_{y \rightarrow y'} 
			+ \sum_{y' \in \calC_Y} \right) (y' - Y) \\[4ex]
		&=& \displaystyle \sum_{y \rightarrow y' \in \calR} w_{y \rightarrow y'} (y' - y)\,,
	\end{array}
	\]
where
\begin{equation}\label{eq:w1}
w_{y \rightarrow y'} = v_{y \rightarrow y'}\,,
\quad
\text{if}\ y \rightarrow y' \in \calR^*_c \backslash \calR^*_{Y}\,,
\end{equation}
\begin{equation}\label{eq:w12}
w_{y \rightarrow y'} = \frac{v_{y \rightarrow y'}}{2}\,,
\quad
\text{if}\ y \rightarrow y' \in \calR^*_Y \cap \calR^*_c\,,
\end{equation}
\begin{equation}\label{eq:w2}
w_{y \rightarrow Y} = 
\left\{
\begin{array}{rl}
\left(
\displaystyle 
\sum_{y \rightarrow y' \in \calR^*_Y \backslash \calR^*_c} + \frac{1}{2} \sum_{y \rightarrow y' \in \calR^*_Y \cap \calR^*_c} 
\right)
v_{y \rightarrow y'} + 1\,, & \quad \text{if}\ y \in \calC_{Y} \\[1ex]
\left(
\displaystyle 
\sum_{y \rightarrow y' \in \calR^*_Y \backslash \calR^*_c} + \frac{1}{2} \sum_{y \rightarrow y' \in \calR^*_Y \cap \calR^*_c} 
\right)
v_{y \rightarrow y'}\,, & \quad \text{if}\ y \notin \calC_{Y}\,,
\end{array}
\right.
\end{equation}
and, similarly,
\begin{equation}\label{eq:w32}
w_{Y \rightarrow y'} = 
\left\{
\arraycolsep=0.25em
\begin{array}{rl}
\left(
\displaystyle 
\sum_{y \rightarrow y' \in \calR^*_Y \backslash \calR^*_c} + \frac{1}{2} \sum_{y \rightarrow y' \in \calR^*_Y \cap \calR^*_c} 
\right) 
v_{y \rightarrow y'} + 1\,, & \quad \text{if}\ y' \in \calC_{Y} \\[1ex]
\left(
\displaystyle 
\sum_{y \rightarrow y' \in \calR^*_Y \backslash \calR^*_c} + \frac{1}{2} \sum_{y \rightarrow y' \in \calR^*_Y \cap \calR^*_c} 
\right)
v_{y \rightarrow y'}\,, & \quad \text{if}\ y' \notin \calC_{Y}\,.
\end{array}
\right.
\end{equation}
Since $G^*$ is obtained from $G$ by the removal of a single intermediate $Y$, every reaction in $\calR$ is of the form $y \rightarrow y'$, $y \rightarrow Y$ or $Y \rightarrow y'$ for some $y \rightarrow y' \in \calR^*$, or of the form $y \rightarrow Y$ or $Y \rightarrow y$ for some $y \in \calC_Y$. Thus, (\ref{eq:w1})--(\ref{eq:w32}) above yield $w_{y \rightarrow y'} > 0$ for every $y \rightarrow y' \in \calR$, and we conclude that $G$ is consistent.

($\Leftarrow$) Now suppose $G$ is consistent, so that there exist $w_{y \rightarrow y'} > 0$, {for all } $y \rightarrow y' \in \calR$, such that
	\begin{equation}\label{eq:quase}
		\sum_{y \rightarrow y' \in \calR} w_{y \rightarrow y'}(y' - y) = 0\,.
	\end{equation}
We partition the set $\calR$ of reactions of $G$ as the (disjoint) union
	\[
		\calR = \calR^*_c \cup \calR_{\rightarrow Y} \cup \calR_{Y \rightarrow}\,,
	\]
where
$\calR^*_c$ is defined as in the first part of the proof, $\calR_{\rightarrow Y}$ is the subset of $\calR$ of reactions having $Y$ as a product, and $\calR_{Y \rightarrow}$ is the subset of $\calR$ of reactions having $Y$ as a reactant. Observe that $Y$ is linearly independent with each complex in $\calC^*$. Combining all coefficients of $Y$ in (\ref{eq:quase}), we obtain
\[
	\sum_{y \rightarrow Y \in \calR_{\rightarrow Y}} w_{y \rightarrow Y}
	- \sum_{Y \rightarrow y' \in \calR_{Y \rightarrow}} w_{Y \rightarrow y'}
	= 0\,,
\]
and so
\begin{equation}\label{eq:quase2}
	\sum_{Y \rightarrow y' \in \calR_{Y \rightarrow}} w_{Y \rightarrow y'} y'
	- \sum_{y \rightarrow Y \in \calR_{\rightarrow Y}} w_{y \rightarrow Y} y
	+ \sum_{y \rightarrow y' \in \calR^*_c} w_{y \rightarrow y'} (y' - y)
	= 0\,.
\end{equation}
Set
\[
	V :=
	\sum_{y \rightarrow Y \in \calR_{\rightarrow Y}} w_{y \rightarrow Y}
	= \sum_{Y \rightarrow y' \in \calR_{Y \rightarrow}} w_{Y \rightarrow y'}\,.
\]
We have
\[
	\sum_{Y \rightarrow y' \in \calR_{Y \rightarrow}} w_{Y \rightarrow y'} y'
	= \sum_{y \rightarrow Y \in \calR_{\rightarrow Y}} \sum_{Y \rightarrow y' \in \calR_{Y \rightarrow}} \frac{w_{y \rightarrow Y}w_{Y \rightarrow y'}}{V} y'
\]
and
\[
	\sum_{y \rightarrow Y \in \calR_{\rightarrow Y}} w_{y \rightarrow Y} y
	= \sum_{y \rightarrow Y \in \calR_{\rightarrow Y}} \sum_{Y \rightarrow y' \in \calR_{Y \rightarrow}} \frac{w_{y \rightarrow Y}w_{Y \rightarrow y'}}{V} y\,.
\]
Plugging these last two identities into (\ref{eq:quase2}), we may rewrite it as
\[
	\sum_{y \rightarrow y' \in \calR^*} v_{y \rightarrow y'}(y' - y) = 0\,,
\]
where
\[
v_{y \rightarrow y'} := \begin{cases} w_{y \rightarrow y'}\, & \text{if}\ y \rightarrow y' \in \calR^*_c \backslash \calR^*_{\rightarrow}\,, \\
 w_{y \rightarrow y'} + \frac{w_{y \rightarrow Y}w_{Y \rightarrow y'}}{V}\, & \text{if}\ y \rightarrow y' \in \calR^*_c \cap \calR^*_{\rightarrow}\,, \\
  \frac{w_{y \rightarrow Y}w_{Y \rightarrow y'}}{V}\, & \text{if}\ \calR^*_{\rightarrow} \backslash \calR^*_c\,.
\end{cases}
\]
In particular, $v_{y \rightarrow y'} > 0$ for every $y \rightarrow y' \in \calR^*$, showing that $G^*$ is consistent.
\end{proof}

\noindent {\textit{Proof of Theorem  \ref{thm:main}\ilr{iii}.}}
Let $G_p := G$ and, for $i = p, \ldots, 1$, let $G_{i-1}$ be the reaction network obtained from $G_i$ by the removal of the set of intermediates $\{Y_i\}$. By Lemma \ref{lemm:intermediateinduction}, $G_0 = G^*$. Iterating Lemma \ref{lemm:consistencypreserved}, we conclude that $G^*$ is consistent if, and only if $G$ is consistent.
\mybox

\subsection{Catalysts}\label{subsec:catalysts_proof}

Suppose $G^* = (\calS^*, \calC^*, \calR^*)$ is the reduction of $G = (\calS, \calC, \calR)$ by the removal of a set of catalysts $\calE$. Let $G_\calE = (\calS_\calE, \calC_\calE, \calR_\calE)$ be the subnetwork of $G$ implied by $\calE$, and write
\[
	\calS^* = \{S^*_1, \ldots, S^*_n\}\,, \qquad \calS_\calE = \{E^{a}_1, \ldots, E^{a}_{q_a}\}\,,
\quad \text{and} \quad
	\calE \backslash \calS_\calE = \{E^{u}_1, \ldots, E^{u}_{q_u}\}\,.
\]
Thus
\[
	\calS = \{S^*_1, \ldots, S^*_n, E^{a}_1, \ldots, E^{a}_{q_a}, E^{u}_1, \ldots, E^{u}_{q_u}\}\,.
\]
These are the orderings we shall assume on the species whenever working with the stoichiometric matrices or stoichiometric subspaces of $G$, $G^*$ or $G_\calE$.

\subsubsection*{Conservation Laws}

\begin{lemma}\label{lemm:catalystcls}
	Suppose $G^* = (\calS^*, \calC^*, \calR^*)$ is the reduction of $G = (\calS, \calC, \calR)$ by the removal of a set of catalysts $\calE$. Then
	\[
		\Gamma^\perp = (\Gamma^*)^\perp \times \Gamma_\calE^\perp \times \r^{q_u}
		\subseteq
		\r^{n + q_a + q_u}\,.
	\] 
\end{lemma}

\begin{proof}
	Write
\[
	\calR^* = \{R^*_1, \ldots, R^*_m\}\,,
\]
and set $\calR_S := \calR \backslash \calR_\calE$. For each $j \in [m]$, let $R^{(j)}_1, \ldots, R^{(j)}_{k_j} \in \calR_S$ be the reactions of $G$ from which $R^*_j$ is obtained by removing the catalysts from both reactant and product in the construction of $G^*$. Write
\[
	\calR_\calE = \{R^\calE_1, \ldots, R^\calE_{m_\calE}\}\,.
\]
Thus
\[
	\calR = \{R^{(1)}_1, \ldots, R^{(1)}_{k_1}, \ldots, R^{(m)}_1, \ldots, R^{(m)}_{k_m}, R^\calE_1, \ldots, R^\calE_{m_\calE}\}\,.
\]
With these orderings on $\calR$, $\calR^*$ and $\calR_\calE$, we may express the stoichiometric matrix $N$ of $G$ as
\begin{equation}\label{eq:catalystsstoichiometry}
	N =
	\left[\begin{array}{cc}
		N^\prime & 0 \\
		0 & N_\calE \\
		0 & 0
	\end{array}\right]\,,
\end{equation}
where $N'$ has $n$ rows, $k_1 + \cdots + k_m$ columns, and has the property that the columns corresponding to $R^{(j)}_1, \ldots, R^{(j)}_{k_j}$ are equal to the $j^{th}$ column of $N^*$, for $j = 1, \ldots, m$, where $N_\calE$ is the stoichiometric matrix of $G_\calE$, and where the bottom $q_u$ rows are zero.

Given $\omega^* \in \r^n$, we have $(\omega^*)^T N^* = 0$ if, and only if $(\omega^*)^T N' = 0$. Hence, given
\[
	\omega = (\omega^*, \omega_\calE, x) \in \r^{n + q_a + q_u}\,,
\]
we have $\omega^T N = 0$ if, and only if $(\omega^*)^T N^* = 0$, and $\omega_\calE^T N_\calE = 0$. This proves the lemma.
\end{proof}

\noindent {\textit{Proof of Theorem  \ref{thm:main2}\ilr{iv}.}}
{($\Rightarrow$) If $\omega = (\omega^*, \omega_\calE, x)$ is a strictly positive conservation law of $G$, then $\omega^*$ is a strictly positive conservation law of $G^*$ by Lemma \ref{lemm:catalystcls}.}

{($\Leftarrow)$  If $\omega^*$ is a strictly positive conservation law of $G^*$, then choose any $x\in \mathbb{R}^{q_u}_{>0}$ and a strictly positive vector $\omega_\calE\in \Gamma_\calE^\perp$, which exists since $G_\calE$ is conservative. Then   
$\omega = (\omega^*, \omega_\calE,x)$ is a strictly positive conservation law of $G$ by Lemma~\ref{lemm:catalystcls}.}
\mybox

\subsubsection*{Siphons}

\begin{lemma}\label{lemm:catalystsiphons_all}
	Suppose $G^* = (\calS^*, \calC^*, \calR^*)$ is the reduction of $G = (\calS, \calC, \calR)$ by the removal of a set of catalysts $\calE$. Let $\Sigma$ be a minimal siphon of $G$. Then one of the three possibilities below is true.
	\begin{itemize}
		\item[{{\em(}i\,{\em)}}] $\Sigma \subseteq \calS^*$, and it is a minimal siphon of $G^*$.
		
		\item[{{\em(}ii\,{\em)}}] $\Sigma \subseteq \calS_\calE$, and it is a minimal siphon of $G_\calE$.
		
		\item[{{\em(}iii\,{\em)}}] $\Sigma = \{E\}$ for some $E \in \calE \backslash \calS_\calE$.
	\end{itemize}
\end{lemma}

\begin{proof}
	Suppose $\Sigma \cap (\calE \backslash \calS_\calE) \neq \varnothing$. Pick any $E \in (\calE \backslash \calS_\calE)$. Then $E$ appears as a reactant in every reaction in which it also appears as a product. We conclude that $\{E\}$ is a siphon, which must then be minimal. It follows that ({\em iii}\,) holds.
	
	Now suppose $\Sigma \cap (\calE \backslash \calS_\calE) = \varnothing$. We have two possibilities. 
	
	If $\Sigma_\calE := \Sigma \cap \calS_\calE \neq \varnothing$, then it is a siphon of $G$. Indeed, pick any $S' \in \Sigma_\calE$, and let $y \rightarrow y' \in \calR$ be any reaction having $S'$ as one of its products. Since $\Sigma$ is a siphon of $G$, $y \rightarrow y'$ must have one of its reactants $S$ in $\Sigma$. If $y \rightarrow y' \notin \calR_\calE$, then $S'$ is also a reactant in $y \rightarrow y'$, and we may assume without loss of generality that $S = S'$. If  $y \rightarrow y' \in \calR_\calE$, then we have $y = S$ and $S \in \calS_\calE$. In either case, $y \rightarrow y'$ has a reactant $S$ in $\Sigma_\calE$. This shows $\Sigma_\calE$ is a siphon of $G_\calE$. By the minimality assumption, we must have $\Sigma = \Sigma_\calE \subseteq \calS_\calE$. Since every siphon of $G_\calE$ is also a siphon of $G$, we conclude that $\Sigma_\calE$ must be a minimal siphon of $G_\calE$.
	
	If $\Sigma \cap \calE = \varnothing$, then $\Sigma \subseteq \calS^*$. It follows from the construction of $G^*$ that $\Sigma$ is a minimal siphon of $G^*$.
\end{proof}

\subsubsection*{Drainable and Self-Replicable Siphons}

\medskip
{\textit{Proof of Theorem  \ref{thm:main2}\ilr{i}.}}
$(\Leftarrow)$ In virtue of (C1), any drainable (respectively, self-replicable) siphon of $G^*$ is also a drainable (respectively, self-replicable) siphon of $G$.

$(\Rightarrow)$ Suppose $\Sigma$ is a drainable or self-replicable siphon of $G$. It is evident from Definition \ref{def:d_or_sr} that any siphon of $G$ contained in $\Sigma$ is also drainable {or self-replicable}. Therefore, we may assume without loss of generality that $\Sigma$ is minimal. By Lemma \ref{lemm:catalystsiphons_all}, either $\Sigma \subseteq \calS^*$, or $\Sigma \subseteq \calS_\calE$, or $\Sigma = \{E\}$ for some $E \in \calE \backslash \calS_\calE$. 

If $\Sigma = \{E\}$ for some $E \in \calE \backslash \calS_\calE$, then the row of $N$ corresponding to $E$ is identically zero, and so the vector of $\rplus^{q + p}$ having its entry corresponding to $E$ equal to $1$ and all other entries equal to zero is a P-semiflow supported in $\Sigma$. In particular, $\Sigma$ is not critical, therefore neither drainable nor self-replicable {by Proposition~\ref{prop:connection}\olr{i}.}

If $\Sigma \subseteq \calS_\calE$, then it follows from (C2) that $\Sigma$ cannot be drainable or self-replicable either.

So, it must be the case that $\Sigma \subseteq \calS^*$. 
{%
Consider a reaction $y\rightarrow y'$ in $\calR^*$ and  the reaction $\widetilde{y}\rightarrow \widetilde{y}'$ in $\calR$ giving rise to it. 
Then the $i$-th coordinate of the vectors $y'-y$ and $\widetilde{y}'-\widetilde{y}$ agree for all $i\in [q]$. Using this observation,}
we conclude that $\Sigma^* := \Sigma \cap \calS^* = \Sigma$ is a drainable or self-replicable siphon of $G^*$.
\mybox

\subsubsection*{Consistency}

{\textit{Proof of Theorem  \ref{thm:main2}\ilr{iii}.}}
	We write the stoichiometric matrix $N$ of $G$ as in the proof of Lemma \ref{lemm:catalystcls}.
	
	First suppose that $G$ is consistent, and let $v \gg 0$ be such that $Nv = 0$. Thus, $N'v' = 0$, where
$v' := (v_1, \ldots, v_{k_1 + \cdots + k_m}) \gg 0$.	Defining $v^* \in \r^m$ by setting
	\[
		v^*_j := v_{k_1 + \cdots + k_{j-1} + 1} + \cdots + v_{k_1 + \cdots + k_{j-1} + k_j}\,,
	\]
	we then get $v^* \gg 0$ and $N^*v^* = 0$, showing that $G^*$ is consistent.
	
	Now suppose $G^*$ is consistent and $G_\calE$ is conservative. Let $v^* \gg 0$ be any vector such that $N^*v^* = 0$. Set
	\[
		v_j' := \frac{1}{k_j}(v_j, \ldots, v_j) \in \r^{k_j}\,,
		\quad
		j = 1, \ldots, m\,,
	\]
	and then set
	\[
		v' := (v_1', \ldots, v_m') \in \r^{k_1 + \cdots + k_m}\,.
	\]
	Then $N'v' = 0$. Since $G_\calE$
	{does not have drainable siphons and is conservative, it follows from Proposition~\ref{prop:persistent2bounded_persistent} that $G_\calE$ is consistent. }Let $v_\calE \gg 0$ be such that $N_\calE v_\calE = 0$. Setting $v := (v', v_\calE)$, we have $v \gg 0$, and $Nv = 0$, proving that $G$ is consistent.
\mybox

\subsection{Uniqueness of The \Primitive\ Reduction}\label{subsec:unique_primitive_reduction}

To prove Theorem \ref{thm:primitive}, we will use induction on the number of species. We start with a few observations and auxiliary results.

In this subsection we will use the following notation. Given a reaction network $G = (\calS, \calC, \calR)$ and a set $\calA \subseteq \calS$ of intermediates or catalysts of $G$, we will denote by $G^*_\calA = (\calS^*_\calA, \calC^*_\calA, \calR^*_\calA)$ the reaction network obtained from $G$ by the removal of $\calA$ (as a set of intermediates or catalysts, whichever happens to be the case). Given another set $\calB \subseteq \calS$ of intermediates (respectively, catalysts) of $G$, note that $\calB \backslash \calA$ is either empty, or else also a set of intermediates (respectively, catalysts) of $G^*_\calA$. We then denote by $G^*_{\calA\calB} = (\calS^*_{\calA\calB}, \calC^*_{\calA\calB}, \calR^*_{\calA\calB})$ the reaction network obtained from $G^*_\calA$ by the removal of $\calB \backslash \calA$.

\begin{lemma}\label{lemm:primitive1and2of3}
	Given a reaction network $G = (\calS, \calC, \calR)$, suppose $\calA, \calB \subseteq \calS$ are two sets of intermediates or two sets of catalysts of $G$. Let $\calD := \calA \cup \calB$. Then $G^*_{\calD} = G^*_{\calA\calB} = G^*_{\calB\calA}$.
\end{lemma}

\begin{proof}
	If $\calA$ and $\calB$ are both sets of intermediates, then the result follows from Lemma \ref{lemm:intermediateinduction}. Removing first the intermediates in $\calA$ one at a time, then removing the intermediates in $\calB \backslash \calA$ yields $G^*_{\calA\calB}$. The analogue procedure starting with the intermediates in $\calB$ yields $G^*_{\calB\calA}$. One then concludes by the same lemma that $G^*_{\calA\calB} = G^*_{\calB\calA} = G^*_\calD$.
	
	Now suppose $\calA$ and $\calB$ are both sets of catalysts. Then both $\calR^*_{\calA\calB}$ and $\calR^*_{\calB\calA}$ consist of the reactions
		\[
			\sum_{i\colon S_i \notin \calD} \alpha_iS_i \longrightarrow \sum_{i\colon S_i \notin \calD} \alpha'_iS_i
		\]
		such that
		\[
			\sum_{i = 1}^n \alpha_iS_i \longrightarrow \sum_{i = 1}^n \alpha'_iS_i
		\]
		belongs to $\calR$, and $\alpha_{i_0} > 0$ or $\alpha'_{i_0} > 0$ for some $i_0 \in [n]$ such that $S_{i_0} \notin \calD$. This shows $\calR^*_{\calA\calB} = \calR^*_{\calB\calA} = \calR^*_\calD$, establishing the result.
\end{proof}

Finally, the removal of a set of catalysts also commutes with the removal of a set of intermediates, in the following sense.

\begin{lemma}\label{lemm:primitive3of3}
	Let $G = (\calS, \calC, \calR)$ be a reaction network, $\calY \subseteq \calS$ be a set of intermediates, and $\calE \subseteq \calS$ be a set of catalysts. Then $G^*_{\calY\calE} = G^*_{\calE\calY}$.
\end{lemma}

\begin{proof}
	Let $\calR(\calY)$ be the subset of reactions $c \rightarrow c' \in \calR$ having some intermediate in $\calY$ as a reactant or product. It follows directly from property (I2) of intermediates that $\calR(\calY)$ is the subset of reactions $c \rightarrow c' \in \calR$ which appear in some reaction path
	\[
		y \longrightarrow Y^{(1)} \longrightarrow \cdots \longrightarrow Y^{(k)} \longrightarrow y'
	\]
	such that $y, y' \in \calC \backslash \calY$ and $Y^{(1)}, \ldots, Y^{(k)} \in \calY$. Let $\calR(\calE)$ be the subset of reactions $c \rightarrow c' \in \calR$ having some catalyst in $\calE$ as both reactant and product. Observe that $\calR(\calY) \cap \calR(\calE) = \varnothing$. Thus, both $\calR^*_{\calY\calE}$ and $\calR^*_{\calE\calY}$ consist of the set of reactions $y \rightarrow y'$ such that $y \rightarrow y' \in \calR \backslash (\calR(\calY) \cup \calR(\calE) \cup \calR_\calE)$, or
	\[
		y \longrightarrow Y^{(1)}, Y^{(1)} \longrightarrow Y^{(2)}, \ldots, Y^{(k-1)} \longrightarrow Y^{(k)}, Y^{(k)} \longrightarrow y' \in \calR(\calY)
	\]
	for some $y, y' \in \calC \backslash \calY$ and $Y^{(1)}, \ldots, Y^{(k)} \in \calY$, or
	\[
		y \longrightarrow y' = \sum_{i\colon S_i \notin \calE} \alpha_iS_i \longrightarrow \sum_{i\colon S_i \notin \calE} \alpha'_iS_i
	\]
	for some
	\[
			\sum_{i = 1}^n \alpha_iS_i \longrightarrow \sum_{i = 1}^n \alpha'_iS_i
	\]
	belonging to $\calR(\calE)$.
\end{proof}

\noindent {\em {Proof of Theorem \ref{thm:primitive}}.}
	We use induction on the number of species. A reaction network with zero species (the empty network) is already \primitive, so, in this case, the result holds vacuously.
	
	Now suppose the result holds for reaction networks with up to $n \geqslant 0$ species, and let $G = (\calS, \calC, \calR)$ be a reaction network with $|\calS| = n + 1$ species. If $G$ is already \primitive, then it is automatically its unique \primitive\ reduction, in which case we have nothing left to prove. So, we may assume $G$ is not \primitive. 
	
	Let $\calA, \calB \subseteq \calS$ be sets of intermediates or catalysts of $G$ such that $\calA \neq \calB$. By the induction hypothesis, $G^*_\calA$ and $G^*_\calB$ have unique \primitive\ reductions, respectively, $G^{**}_\calA$ and $G^{**}_\calB$. We want to show that $G^{**}_\calA = G^{**}_\calB$.
	
	Let $G^{**}_{\calA\calB}$ (respectively, $G^{**}_{\calB\calA}$) be the \primitive\ reduction of $G^*_{\calA\calB}$ (respectively, $G^*_{\calB\calA}$). Note that $G^{**}_{\calA\calB} = G^{**}_\calA$ and $G^{**}_{\calB\calA} = G^{**}_\calB$. By Lemmas \ref{lemm:primitive1and2of3} and \ref{lemm:primitive3of3}, $G^*_{\calA\calB} = G^*_{\calB\calA}$, and hence $G^{**}_\calA = G^{**}_\calB$.
\hfill\mybox

\appendix
\section{Technical Results}

\subsection{Proof of Proposition~\ref{prop:persistent2bounded_persistent}}\label{app:prop2}

 {\it {Proposition~\ref{prop:persistent2bounded_persistent}\ilr{i}.}} We prove that if $G$ is persistent, then it is bounded-per\-sist\-ent.  Take any $s_0 \gg 0$. If $\omega(s_0) = \varnothing$, then we have nothing to prove. So, suppose $\omega(s_0) \neq \varnothing$. Choose any $s \in \omega(s_0)$, and a sequence $(t_k)_{k \in \n}$ going to infinity in $\rplus$ such that
		\[
			\lim_{k \to \infty} \sigma(t_k, s_0) = s\,.
		\]
		Then
		\[
			s_i 
			= \displaystyle \liminf_{k \to \infty} \sigma_i(t_k, s_0)
			\geqslant \displaystyle \liminf_{t \to \infty} \sigma_i(t, s_0)
			> 0\,, 
				\quad \forall i \in [n]\,.
		\]
		In particular, $s \notin \partial\rplus^n$. Thus, $\omega(s_0) \cap \partial\rplus^n = \varnothing$.
\mybox

\medskip
\noindent
 {\it {Proposition~\ref{prop:persistent2bounded_persistent}\ilr{ii}.}} 
The converse of Proposition~\ref{prop:persistent2bounded_persistent}\olr{i} is not true. However, (\ref{eq:persistent}) holds for bounded trajectories of bounded-persistent networks---hence the terminology.

Since each stoichiometric compatibility class of a conservative network is compact \cite[Appendix 1]{horn--jackson-1972}, every solution of (\ref{eq:crntode}) is bounded. The proof of Proposition~\ref{prop:persistent2bounded_persistent}\olr{ii} then follows 
from the next lemma.

\begin{lemma}\label{lemm:bounded_persistent2persistent}
	Suppose a solution $\sigma(\cdot, s_0)\colon \rplus \rightarrow \rplus^n$ of a bounded-per\-sist\-ent reaction network is bounded. Then
	\begin{equation}\label{eq:persistent2}
		\liminf_{t \to \infty} \sigma_i(t, s_0) > 0\,,
		\quad
		\forall i \in [n]\,.
	\end{equation}
\end{lemma}

\begin{proof}
	Suppose on contrary that
		\[
			\liminf_{t \to \infty} \sigma_{i_0}(t, s_0) = 0
		\]
		for some $i_0 \in [n]$. Then
		\[
			\lim_{k \to \infty} \sigma_{i_0}(t_k, s_0) = 0
		\]
		along some sequence $(t_k)_{k \in \n}$ going to infinity in $\rplus$. In virtue of boundedness, by passing into a subsequence, if necessary, we may assume without loss of generality that $(\sigma(t_k, s_0))_{k \in \n}$ converges, say,
		\[
			\lim_{k \to \infty} \sigma(t_k, s_0) = s_\infty\,.
		\]
		We have $s_\infty \in \omega(s_0)$ by definition. But since the $i_0^{th}$ coordinate of $s_\infty$ is zero, we conclude that $s_\infty \in \partial\rplus^n$ also. This contradicts the bounded-persistence hypothesis that $\omega(s_0) \cap \partial\rplus^n = \varnothing$. Thus, (\ref{eq:persistent2}) must hold.
\end{proof}

 \medskip
 \noindent
 {\it {Proposition~\ref{prop:persistent2bounded_persistent}\ilr{iii}.}} 
 See \cite[Theorem 1]{angeli--deleenheer--sontag-2007c}.
\mybox

\medskip
\noindent
 {\it {Proposition~\ref{prop:persistent2bounded_persistent}\ilr{iv}.} }
The same argument as in \cite[Theorem 6.2]{deshpande--gopalkrishnan-2014} works under our weaker assumptions on the reaction rates.
	\mybox

\medskip
\noindent
 {\it {Proposition~\ref{prop:persistent2bounded_persistent}\ilr{v}.} }
We define the {\em zero coordinate set} of a point $s \in \rplus^n$, with respect to some given reaction network $G$, as the set
	\[
		Z(s) := \{S_i \in \calS\, | \ s_i = 0\} = \calS \backslash \supp s\,.
	\]
Thus, a point $s \in \rplus^n$ is a boundary steady state if, and only if $Z(s) \neq \varnothing$.

Let $s_0$ be a boundary steady state of $G$. 
	{By Lemma \ref{lemm:bss1} below and our hypothesis, the zero coordinate set $Z(s_0)$ of $s_0$ is a noncritical siphon.
	It follows by the equivalence between items 1.\ and 3.\ in \cite[Theorem 3.7]{deshpande--gopalkrishnan-2014} that $(s_0 + S) \cap \rplus^n\neq \varnothing$.}

The next lemma was proved in \cite{shiu--sturmfels-2010} for mass-action kinetics. The same argument holds under (r2), and we provide the details for the sake of completeness.
 
\begin{lemma}\label{lemm:bss1}
	Let $G$ be a reaction network. If $s_0$ is a boundary steady state, then $Z(s_0)$ is a siphon.
\end{lemma}

\begin{proof}
	Pick any $S_{i} \in Z(s_0)$.
Consider the set $\calJ_{i}$ of indices $j \in [m]$ such that $R_j$ is a reaction having $S_{i}$ as one of its products, but not one of its reactants; that is,
	\begin{equation*}
		\calJ_{i} := \{j \in [m]\,|\ \alpha'_{ij} > 0 \ \text{and} \ \alpha_{ij} = 0\}\,.
	\end{equation*}
	If  $\calJ_i \neq \varnothing$, we need to show that  $Z(s_0)$ contains some species in the reactant of each $R_j$ such that   $j\in \calJ_{i}$.
	Since $s_0$ is a steady state, we have
	\begin{equation}\label{eq:ss}
		 \sum_{j = 1}^m (\alpha'_{ij} - \alpha_{ij}) r_j(s_0) = 0\,.
	\end{equation}	
	For each $j \notin \calJ_{i}$,  we either have $\alpha_{ij} > 0$ (in which case $r_j(s_0) = 0$ by (r2) since $(s_0)_i=0$) or
	$\alpha'_{ij} = \alpha_{ij}=0$.
	Hence the sum in (\ref{eq:ss}) can be simplified as
	\[
		\sum_{j \in \calJ_{i}} \alpha'_{ij} r_j(s_0) = 0\,.
	\]
	Since $\alpha'_{ij} > 0$ for every $j \in \calJ_i$ by construction, we conclude that
		$r_j(s_0) = 0$ for all $j \in \calJ_{i}$.
	It then follows from (r2) that $\alpha_{i(j)j} > 0$ for some $i(j) \in [n]$ such that $S_{i(j)} \in Z(s_0)$, that is, one of the reactants of $R_j$ belongs to $Z(s_0)$ for each $j \in \calJ_{i}$. This completes the proof that $Z(s_0)$ is a siphon.	
\end{proof}

\subsection{Drainable and Self-Replicable Siphons}\label{subsec:equivalence}

The next result shows that the concepts of drainable and self-replicable {sets} in Definition \ref{def:d_or_sr} are, respectively, equivalent to the concepts of drainable and self-replicable {sets} in \cite[Definition 3.1]{deshpande--gopalkrishnan-2014} {(called here {\em DG-drainable} and {\em DG-self-replicable})}.

Given a reaction network $G = (\calS, \calC, \calR)$, we define a {\em $G$-reaction pathway} to be any sequence $y(0), y(1), \ldots, y(k) \in \rplus^n$ such that
\begin{align}
		y(0) &= y_1 + w_1\,,  \nonumber \\
		y(j)  & = y_j' + w_j = y_{j+1} + w_{j+1}\,,
		\quad
		j = 1, \ldots, k-1\,, \label{eq:DG}
	\\
		y(k) &= y_k' + w_k\,, \nonumber 
\end{align}
	for some $y_1, y_1', w_1, \ldots, y_k, y_k', w_k \in \rplus^n$ such that $y_1 \rightarrow y_1, \ldots, y_k \rightarrow y_k' \in \calR$. 
{Note that
\begin{equation}\label{eq:Gpathway}
			y(k) - y(0) 
			=   \sum_{j=1}^k \big( y(j) - y(j-1) \big) =   \sum_{j=1}^k (y_j' + w_j - y_j - w_j) 
			=   \sum_{j=1}^k (y_j' - y_j)\,.
\end{equation}}

	A nonempty subset $\Sigma \subseteq \calS$ is said to be {\em DG-drainable} (respectively, {\em DG-self-replicable}) if there exists a $G$-reaction pathway $y(0), y(1), \ldots, y(k)$ such that $\big(y(k) - y(0) \big)_i < 0$ (respectively, $\big(y(k) - y(0) \big)_i > 0$), for every $i \in [n]$ such that $S_i \in \Sigma$.

\begin{proposition}\label{prop:understood}
	Let $G = (\calS, \calC, \calR)$ be a reaction network. A {subset} of $\calS$ is drainable (respectively, self-replicable) if, and only if it is DG-drainable (respectively, DG-self-replicable).
\end{proposition}

\begin{proof}
{	$(\Leftarrow)$ 
	Follows from \eqref{eq:Gpathway} and Definition \ref{def:d_or_sr}.
	
	$(\Rightarrow)$ Let $y_1 \rightarrow y_1', \ldots, y_k \rightarrow y_k' \in \calR$ be any sequence of reactions. 
	Define, iteratively,
$$w_1 := y_2+\dots+y_k, \quad\text{and}\quad w_{j+1} := y_j' + w_j - y_{j+1}, \quad j=1,\dots,k-1.$$
		By construction, $\omega_j\in  \rplus^n$ for all $j=1,\dots,k$, and $y_{j+1}+w_{j+1} = y'_j + w_j$ for all $j=1,\dots,k-1$. We can construct a G-reaction pathway $y(0), y(1), \ldots, y(k) \in \rplus^n$ from $y_1, y_1', w_1, \ldots, y_k, y_k', w_k$ using \eqref{eq:DG}. The implication now follows again from \eqref{eq:Gpathway} and Definition \ref{def:d_or_sr}.}
\end{proof}

\subsubsection*{Acknowledgements}
	Elisenda Feliu, Michael Marcondes de Freitas and Carsten Wiuf acknowledge funding from the Danish Research Council of Independent Research. We would also like to thank the reviewers for their helpful comments and suggestions, which have greatly improved this work.

\end{document}